\documentclass[10pt]{article}
\usepackage[
  margin=3cm,
  includefoot,
  footskip=30pt,
]{geometry}

\usepackage[T1]{fontenc}
\usepackage{textcomp}
\usepackage{libertinus}

\usepackage{url}
\usepackage{fullpage}
\usepackage{diagbox}

\usepackage{natbib}

\usepackage{amsthm,amsfonts,amsmath,amssymb,mathtools,epsfig,color,float,graphicx,verbatim,bm,bbm}
\usepackage{aliascnt}
\usepackage{enumerate}
\usepackage{enumitem}
\usepackage{wrapfig}
\usepackage{subcaption}
\usepackage[export]{adjustbox}
\usepackage{nicefrac}
\usepackage{hhline}
\usepackage{multicol}
\usepackage{multirow}
\usepackage[dvipsnames]{xcolor}

\usepackage[algo2e,ruled,vlined,linesnumbered]{algorithm2e}

\definecolor{linkblue}{HTML}{3366CC}

\usepackage{hyperref}
\hypersetup{
  colorlinks = true,
  urlcolor   = linkblue,
  linkcolor  = linkblue,
  citecolor  = linkblue,
  pdftitle   = {Optimal Algorithms for
Nonsmooth Nonconvex Stochastic Optimization with Constraints},
  pdfauthor  = {Gon Buzaglo and Guy Kornowski}
}
\usepackage[nameinlink,capitalise,noabbrev]{cleveref}

\allowdisplaybreaks
\setlist{leftmargin=*}

\newtheorem{theorem}{Theorem}[section]

\newaliascnt{proposition}{theorem}
\newtheorem{proposition}[proposition]{Proposition}
\aliascntresetthe{proposition}

\newaliascnt{lemma}{theorem}
\newtheorem{lemma}[lemma]{Lemma}
\aliascntresetthe{lemma}

\newaliascnt{corollary}{theorem}
\newtheorem{corollary}[corollary]{Corollary}
\aliascntresetthe{corollary}

\newaliascnt{assumption}{theorem}

\aliascntresetthe{assumption}

\newaliascnt{definition}{theorem}
\newtheorem{definition}[definition]{Definition}
\aliascntresetthe{definition}

\newaliascnt{remark}{theorem}
\newtheorem{remark}[remark]{Remark}
\aliascntresetthe{remark}

\newaliascnt{fact}{theorem}
\newtheorem{fact}[fact]{Fact}
\aliascntresetthe{fact}

\newaliascnt{example}{theorem}

\aliascntresetthe{example}

\newaliascnt{conjecture}{theorem}

\aliascntresetthe{conjecture}

\newtheorem*{question*}{Question}

\newtheorem*{bigquestion*}{Big Question}
\newtheorem*{definition*}{Definition}
\newtheorem*{keyquestion*}{Key Question}

\crefname{theorem}{theorem}{theorems}
\Crefname{theorem}{Theorem}{Theorems}
\crefname{lemma}{lemma}{lemmas}
\Crefname{lemma}{Lemma}{Lemmas}
\crefname{proposition}{proposition}{propositions}
\Crefname{proposition}{Proposition}{Propositions}
\crefname{corollary}{corollary}{corollaries}
\Crefname{corollary}{Corollary}{Corollaries}
\crefname{assumption}{assumption}{assumptions}
\Crefname{assumption}{Assumption}{Assumptions}
\crefname{definition}{definition}{definitions}
\Crefname{definition}{Definition}{Definitions}
\crefname{remark}{remark}{remarks}
\Crefname{remark}{Remark}{Remarks}
\crefname{fact}{fact}{facts}
\Crefname{fact}{Fact}{Facts}
\crefname{example}{example}{examples}
\Crefname{example}{Example}{Examples}

\renewcommand{\eqref}[1]{Eq.~(\ref{#1})}

\newcommand{\E}{\mathbb{E}}

\newcommand{\argmin}[1]{\underset{#1}{\mathrm{argmin}}}

\newcommand{\dist}{\mathrm{dist}}

\newcommand{\ba}{\mathbf{a}}

\newcommand{\bx}{\mathbf{x}}

\newcommand{\bw}{\mathbf{w}}

\newcommand{\bg}{\mathbf{g}}

\newcommand{\bu}{\mathbf{u}}

\newcommand{\bv}{\mathbf{v}}

\newcommand{\bz}{\mathbf{z}}

\newcommand{\bh}{\mathbf{h}}

\newcommand{\by}{\mathbf{y}}

\newcommand{\bn}{\mathbf{n}}
\newcommand{\bs}{\mathbf{s}}

\newcommand{\norm}[1]{\|#1\|}

\newcommand{\inner}[1]{\langle#1\rangle}

\newcommand{\conv}{\mathrm{conv}}

\newcommand{\Unif}{\mathrm{Unif}}

\newcommand{\abs}[1]{\left| #1 \right|}
\newcommand{\N}{\mathbb{N}}

\newcommand{\R}{\mathbb{R}}

\newcommand{\eps}{\epsilon}

\providecommand{\Prb}{\mathbb{P}}
\providecommand{\cX}{\mathcal{X}}

\providecommand{\cA}{\mathcal{A}}

\providecommand{\Ball}{\mathbb{B}}
\providecommand{\Sph}{\mathbb{S}}
\providecommand{\prog}{\mathrm{prog}}
\providecommand{\bp}{\mathbf{p}}

\providecommand{\bnabla}{\boldsymbol{\nabla}}
\providecommand{\bell}{\boldsymbol{\ell}}

\providecommand{\proj}{\Pi}

\providecommand{\Reg}{\operatorname{Regret}}
\providecommand{\LMO}{\operatorname{LMO}}
\providecommand{\FW}{\operatorname{FW}}

\providecommand{\wtO}{\widetilde O}
\providecommand{\Res}{\mathcal{R}}

\providecommand{\Gmap}{\mathcal{G}}
\providecommand{\set}[1]{\left\{#1\right\}}

\renewcommand{\argmin}{\operatorname*{arg\,min}}

\renewcommand{\conv}[1][]{%
  \operatorname{conv}%
  \if\relax\detokenize{#1}\relax
  \else
    \left(#1\right)%
  \fi
}
\renewcommand{\inner}[2]{\left\langle #1,#2\right\rangle}
\renewcommand{\eps}{\varepsilon}

\title{Optimal Algorithms for
Nonsmooth Nonconvex\\Stochastic Optimization with Constraints}
\author{Gon Buzaglo
\vspace{3pt}\\Princeton University
\and Guy Kornowski
\vspace{3pt}\\Institute for Advanced Study
}
\date{}

\begin{document}

\maketitle

\begin{abstract}
We present new algorithms for stochastic optimization problems in which the objective is neither smooth nor convex, subject to a convex constraint set.
Our algorithms return a $(\delta,\eps)$-Goldstein stationary point, as recently defined for constrained problems by \citet{liu2024},
with improved complexities, and are the first to achieve optimal convergence rates for this task in several settings.
In particular, we design projection-based and projection-free algorithms that achieve a $\widetilde O(\delta^{-1}\eps^{-3})$ rate using stochastic gradients, which we show to be optimal, and a $\widetilde O(d\delta^{-1}\eps^{-3})$ rate using zeroth-order queries, which is also optimal in its dependence on the dimension $d$ and improves over prior work.
Notably, our lower bound is given in terms of the Frank-Wolfe gap and holds even for smooth problems, which is of independent interest.
These results show that in nonconvex stochastic optimization, the convergence rate is degraded neither by the presence of constraints nor by nonsmoothness.
Our analysis is based on a reduction to online linear optimization, extending a related technique from unconstrained optimization.
\end{abstract}

\section{Introduction}
\label{sec:introduction}
We consider stochastic optimization problems of the form
\begin{equation} \label{eq: problem}
\min_{\bx\in\cX}\,f(\bx):=\E_{\xi\sim\Xi}[F(\bx;\xi)]\,,
\end{equation}
where the objective $f:\R^d\to\R$ is neither smooth nor convex, subject to a closed convex constraint set $\cX\subseteq\R^d$.
Such problems arise naturally in machine learning applications, including in adversarial attacks and robust training
\citep{szegedy2013intriguing,madry2017towards,
wang2022convergencerobustnessadversarialtraining,qi2024visual}, norm-constrained training of deep neural networks
\citep{goodfellow2013maxout,srivastava2014dropout}, and
sharpness-aware training
\citep{foret2021sharpness,chen2023sharpness,mordido2024lookbehind}.
Nonetheless, despite its prevalence, constrained nonsmooth nonconvex optimization is rather challenging to analyze, with most existing work guaranteeing either only asymptotic convergence without explicit finite-time rates, or only applying to deterministic objectives, or both \citep{curtis2012sequential,hare2016proximal,ruszczynski2020convergence,
davis2020stochastic,ermol1998stochastic,hoseini2021filter,grimmerjia2025goldstein}.

Recently, \citet{liu2024} provided the first non-asymptotic convergence rates for algorithms solving problem (\ref{eq: problem}).
To do so, the authors defined two novel stationarity measures suitable for constrained nonsmooth optimization, which generalize the classic gradient mapping norm and the Frank-Wolfe gap, both of which are common measures for \emph{smooth} constrained optimization,
by replacing classical gradients with the Goldstein subdifferential.
The Goldstein subdifferential, 
which is the convex hull of nearby Clarke subdifferentials \citep{goldstein1977,clarke1990}, 
has played a key role in non-asymptotic analyses of (unconstrained) nonsmooth nonconvex optimization problems in recent years following \citet{zhang2020}, since the norm of individual gradients cannot be minimized efficiently in general \citep{kornowski2022oracle}.

While \citet{liu2024} established convergence rates 
for several zeroth-order methods, it is notable that these rates are strictly worse than those achievable for unconstrained optimization when $\cX=\R^d$. In particular, the complexities of the algorithms therein all scale as $O(d^{3/2})$ where $d$ is the problem's dimension, whereas in the analogous unconstrained problem, the optimal complexity is known to be $O(d)$ \citep{kornowski2024}. This raises the natural question, asked by \citet{liu2024}, of whether the complexity achieved in their work can be further improved. In particular, it remained unclear whether incorporating constraints
necessarily makes nonsmooth nonconvex optimization harder to solve.

\subsection{Our Contributions}
In our work, we design and analyze both first-order and zeroth-order algorithms for constrained nonsmooth nonconvex stochastic optimization, extending as well as strictly improving the complexities obtained by \citet{liu2024}.
We establish optimal convergence rates for our algorithms with matching lower bounds, which in particular prove that \emph{constrained nonsmooth optimization is as easy as unconstrained smooth optimization} in terms of query complexity for nonconvex stochastic problems, namely neither constraints nor nonsmoothness degrade the convergence rates.
Our contributions can be summarized as follows:

\begin{itemize}
    \item \textbf{Reduction from Constrained Optimization to Online Learning (\Cref{thm:generic-conversion}):}
    As our main proof technique, we extend the online-to-nonconvex framework of \citet{cutkosky2023} to optimization over a closed convex constraint set.
    Our reduction converts any online learning algorithm that is ``local'' in the sense that its predictions remain close to one another, into a constrained optimization algorithm, and translates the online learner's regret into a local constrained-stationarity guarantee.

    \item \textbf{First-Order Algorithms (\Cref{thm:projection-expectation} and \Cref{thm:lmo-main}):}
    We use our generic reduction to design stochastic first-order algorithms with optimal gradient oracle complexity for problem~(\ref{eq: problem}). In particular, by using classic online learning algorithms such as online gradient descent (OGD) \citep{zinkevich2003online}
    or follow the regularized leader (FTRL) \citep{shalev2007primal} and applying our reduction, we obtain projection-based and projection-free algorithms which, in both cases, require $\widetilde O(\delta^{-1}\eps^{-3})$ stochastic gradients to converge to a constrained approximate-stationary point (formally defined in \Cref{sec:preliminaries}). These are the first gradient-based algorithms for this problem class that are equipped with dimension-free complexity, which as we will further show, is optimal.

    \item \textbf{Zeroth-Order Algorithms (\Cref{thm:two-point-projection,thm:two-point-lmo}):}
    We further discuss zeroth-order algorithms that only query noisy function values, and obtain projection-based and projection-free algorithms which, in both cases, require $\widetilde O(d\delta^{-1}\eps^{-3})$ function-value queries to converge to a constrained approximate-stationary point. This improves over the complexities obtained by \citet{liu2024} by a
$\Omega(\sqrt{d})$ factor in both projection-based and projection-free
settings, and has optimal dependence on both the dimension and the accuracy
parameters, with the recent joint lower bound of
\citet{zhang2026jointlowerboundszerothorder} matching our rate for the
Goldstein residual.
    
\item \textbf{Matching Lower Bounds (\Cref{sec:lower-bounds}):}
Finally, we adapt the lower bound technique of \citet{arjevani2023lower} to the discussed setting. In particular, we prove an $\Omega(\eps^{-4})$ lower bound on the first-order complexity of minimizing the Frank-Wolfe gap in stochastic nonconvex optimization. Our lower bound holds even for smooth objectives and
matches the rates of existing smooth algorithms \citep{reddi2016stochastic}. As we show, this result further yields matching lower
bounds for both the Goldstein Frank-Wolfe gap and the Goldstein residual (to be defined in \Cref{sec:preliminaries}),
establishing the optimality of our convergence rates. 
\end{itemize}

\section{Preliminaries}
\label{sec:preliminaries}
\paragraph{Notation.}
 We use bold-faced letters for vectors, e.g., $\bx\in\R^d$, and
$\norm{\,\cdot\,}$ denotes the Euclidean norm.
We denote $[n]:=\set{1,\ldots,n}$, $\Ball(\bx,r):=\set{\by:\norm{\by-\bx}\le r}$,
and $\dist(\bx,A):=\inf_{\ba\in A}\|\bx-\ba\|$,
and denote by $\conv[\cdot]$ the convex hull operator and by $\operatorname{\Unif}(A)$ the uniform measure over a set $A$.
The projection of a point $\by\in\R^d$ onto a closed convex set $\cX$ is denoted by 
$\proj_\cX(\by):=\argmin_{\bx\in\cX}\norm{\bx-\by}$.
We say that a function $f:\R^d\to\R$ is $L$-Lipschitz if $\abs{f(\bx)-f(\by)}\leq L\norm{\bx-\by}$ for all $\bx,\by\in\R^d$, and $H$-smooth if it is differentiable and $\norm{\nabla f(\bx)-\nabla f(\by)}\leq H\norm{\bx-\by}$.
We follow standard big-$O$ notation with $O\left(\cdot\right),\,\Omega\left(\cdot\right),\,\Theta\left(\cdot\right)$ hiding constants that do not depend on problem parameters and $\tilde O\left(\cdot\right),\,\tilde \Omega\left(\cdot\right),\,\tilde \Theta\left(\cdot\right)$ suppressing logarithmic factors.

\paragraph{Nonsmooth analysis.} 
By Rademacher's theorem, Lipschitz functions are differentiable almost everywhere. Therefore, for a Lipschitz function $f:\R^d\to\R$ and point $\bx\in\R^d$ the Clarke subdifferential \citep{clarke1990} is well-defined as the convex hull of all limit points of $\nabla f(\bx_n)$ over all sequences of differentiable points which converge to $\bx$, formally defined as 
\(
\partial f(\bx):=\conv\{\bg\,:\,\bg=\lim_{n\to\infty}\nabla f(\bx_n),\,\bx_n\to \bx\}\,.
\)
The Clarke subdifferential becomes the subdifferential for a convex $f$, or the gradient if $f$ is continuously differentiable.

\paragraph{Constrained nonsmooth optimization.}

Throughout the paper, we let the constraint set $\cX\subseteq\R^d$ be a nonempty closed convex set.
For $\bx\in\cX$, $\bg\in\R^d$, and $\eta>0$, we define the mapping  $\Gmap_{\eta,\cX}(\bx,\bg):=\frac1\eta\bigl(\bx-\proj_{\cX}(\bx-\eta \bg)\bigr)$,
which coincides with the gradient mapping when $\bg=\nabla f\left(\bx\right)$, whose norm is a standard stationarity criterion for constrained smooth optimization \citep{ghadimi2016accelerated,ghadimi2016minibatch}.

In this paper, following \citet{liu2024}, we let $\delta\geq0$ and consider the generalized
gradient mapping with $\bg\in\partial_\delta f(\bx)$, where
$\partial_\delta f(\bx):=\conv\left(\cup_{\by\in\Ball(\bx,\delta)}\partial f(\by)\right)$
is the Goldstein $\delta$-subdifferential \citep{goldstein1977,zhang2020}.
We define the
$\delta$-Goldstein residual of $f$ at $\bx\in\cX$ as
\[
  \Res^{f,\cX}_\delta(\bx)
  :=
  \min_{\bg\in\partial_\delta f(\bx)}
  \sup_{\eta>0}\norm{\Gmap_{\eta,\cX}(\bx,\bg)}
  =\min_{\bg\in\partial_\delta f(\bx)}
  \dist\bigl(-\bg,\mathcal{N}_\cX(\bx)\bigr)~,
\]
where $\mathcal{N}_\cX(\bx):=\set{\bn\in\R^d:\inner{\bn}{\by-\bx}\le0\text{ for every }\by\in\cX}$ denotes the normal cone of $\cX$ at $\bx\in\cX$, and the equality above is proved in \Cref{lem:projection-residual}. The Goldstein residual above will
serve as a stationarity measure.
Note that this definition reduces to a standard gradient norm bound in the smooth unconstrained setting, since $\mathcal{N}_{\mathbb{R}^d}(\bx)=\{\mathbf{0}\}$ and $\partial_0 f(\bx)=\{\nabla f(\bx)\}$. Moreover, the $\delta$-Goldstein residual relaxes the standard optimality condition
\(
    \mathbf{0}\in \partial f(\bx)+\mathcal{N}_\cX(\bx)
\)
for constrained Lipschitz optimization (see \citealp[Corollary~6.3.9]{borwein2006convex}).

\begin{remark} \label{rem: stronger stat}
\citet{liu2024} considered stationarity with respect to a fixed parameter $\eta>0$, as opposed to taking the supremum over $\eta$ as we do.
Thus, it is notable that our results not only provide improved complexities, but also guarantee a stronger stationarity condition.
\end{remark}

If $\cX$ is compact with some finite diameter $D$, the Frank-Wolfe gap
$\max_{\bu\in\cX}\inner{\nabla f(\bx)}{\bx-\bu}\,
$
is yet another standard optimality criterion, prevalent
in the smooth projection-free optimization literature \citep{lacoste2016convergence}.
For nonsmooth problems, \citet{liu2024} introduced the $\delta$-Frank-Wolfe gap
as
\[
\FW^{f,\cX}_\delta(\bx):=\min_{\bg\in\partial_\delta f(\bx)}\max_{\bu\in\cX}\inner{\bg}{\bx-\bu}\,.
\]
We show that $\FW^{f,\cX}_\delta(\bx)\leq D\cdot\Res^{f,\cX}_\delta(\bx)$ in \Cref{lem:prelim-stability}, namely the two notions differ by at most a diameter factor, and therefore bounding the residual over a compact set implies a bound on the gap.
\Cref{app:projection-geometry} further contains some useful results related to the discussed stationarity measures.

\paragraph{Online linear optimization.}
Online linear optimization (OLO) is a standard framework for sequential
decision-making without stochastic assumptions
\citep{cesa2006prediction,hazan2016introduction,orabona2019modern}, which proceeds as follows:
Over $M\in\mathbb{N}$ rounds, an adversary chooses loss vectors $\bell_m\in\R^d$;
at each round, the learner chooses $\by_m\in C$ from a compact convex set $C\subseteq\R^d$ before observing $\bell_m$, and incurs loss $\inner{\bell_m}{\by_m}$.
We use the following local class of OLO algorithms.

\begin{definition}\label{def:local-online-learner}
Let $\cX$ be nonempty, closed, and convex, and let $r>0$.
An $r$-local OLO algorithm $\cA$ associates with every initialization
$\bx\in\cX$ a nonempty compact convex set $C_{\cA}(\bx)\subseteq \cX\cap\Ball(\bx,r)$, depending only
on $\bx$ and with $\bx\in C_{\cA}(\bx)$,
so that all predictions of $\cA$ initialized at $\bx$ lie in $C_{\cA}(\bx)$.
\end{definition}

In \Cref{sec:projection,sec:lmo}, we obtain such local learners from
standard OLO algorithms with a small computational modification.
Their regret when initialized at $\bx$ is defined as
\[
\Reg_\cA(\bx;\bell_1,\dots,\bell_M)
:=
\max_{\by\in C_{\cA}(\bx)}
\sum_{m=1}^M
\inner{\bell_m}{\by_m-\by}.
\]

\paragraph{Problem setting.}
We consider a stochastic optimization problem associated to an objective
$f:\R^d\to\R$ of the form
\(
    f(\bx)=\E_{\xi\sim\Xi}[F(\bx;\xi)],
\)
subject to a nonempty closed convex set $\cX\subseteq\R^d$.
For simplicity, we assume that $f$ and the stochastic components
$F(\cdot;\xi)$ are differentiable, but do not assume that $f$ is
smooth.\footnote{
The differentiability assumption can be relaxed using standard randomized
smoothing (over an arbitrarily small ball), yielding a differentiable
objective; see \citep{cutkosky2023,davis2022gradient}. We choose to avoid
this slight technicality altogether in our work.}
We assume access to stochastic component gradients: Given any point
$\bx\in\cX$, the oracle samples $\xi\sim\Xi$ independently and returns
\(\nabla F(\bx;\xi)
\),
satisfying 
\[
    \E_\xi[\nabla F(\bx;\xi)]=\nabla f(\bx),
    \qquad
    \E_\xi[\norm{\nabla F(\bx;\xi)}^2]\le G^2.
\]
\section{Constrained Online-to-Nonconvex Conversion}
\label{sec:reduction}

We present a reduction from constrained nonconvex optimization to online learning in \Cref{alg:conversion}. While some earlier online-to-nonconvex results assumed smoothness \citep{agarwal2019efficient,hallak2021regret}, we rely on the nonsmooth conversion of \citet{cutkosky2023}, with subsequent works extending this approach in various directions \citep{zhang2024random,ahn2024understanding,ahn2024adam,pmlr-v267-ahn25a}. These results focus on unconstrained optimization,\footnote{A notable exception is the work of \citet{sahinoglu2025finite}, which extended the online-to-nonconvex framework to Riemannian manifolds. However, their setting is incomparable to ours as it excludes boundaries and therefore does not require incorporating the constraints into the stationarity notion.
} whereas \Cref{alg:conversion} allows general closed convex constraints.
The next result summarizes its guarantee:

\begin{algorithm2e}[t]
\caption{Online-to-nonconvex conversion$(\bx_0,K,M,r,\cA)$}
\label{alg:conversion}
\KwIn{initialization $\bx_0\in\cX$, $r$-local online learner $\cA$, epoch budget $K$, epoch length $M$.}

\For{$k=0,1,\ldots,K-1$}{
  initialize a fresh copy $\cA_k$ of $\cA$ at $\bx_k$\;
  \For{$m=1,\ldots,M$}{
    obtain prediction $\by_{k,m}$ from $\cA_k$\;
    sample $s_{k,m}\sim\Unif([0,1])$\;
    $\bz_{k,m}\gets
    \bx_k+s_{k,m}(\by_{k,m}-\bx_k)$\;
    sample $\xi_{k,m}\sim\Xi$\;
    query $\bg_{k,m}\gets\nabla F(\bz_{k,m};\xi_{k,m})$\;
    send $\bg_{k,m}$ to $\cA_k$\;
  }
  sample $J_k\sim\Unif([M])$ and
  $\bx_{k+1}\gets\by_{k,J_k}$\;
}
sample $I\sim\Unif(\{0,\ldots,K-1\})$\;
\Return{$\bigl(
  \bx_I,
  (\bz_{I,m})_{m=1}^M,
  \frac1M\sum_{m=1}^M\bg_{I,m}
\bigr)$}\;
\end{algorithm2e}

\begin{theorem}
\label{thm:generic-conversion}
Assume $\bx_0\in\cX$ satisfies
$f(\bx_0)-\inf_{\bx\in\cX}f(\bx)\le\Delta<\infty$, and let $\cA$ be an $r$-local online learner. The output of
\Cref{alg:conversion} satisfies
\begin{align}
  \E\!\left[\FW_r^{f,C_\cA(\bx_I)}(\bx_I)\right]
  &\leq\E\left[{\max}_{\by\in C_\cA(\bx_I)}
    \Big\langle \frac{1}{M}{\sum}_{m=1}^M\nabla f(\bz_{I,m}), \bx_I-\by\Big\rangle
  \right]
  \nonumber\\
  &\le
  \frac{\Delta}{K}
    +\frac{1}{KM}{\sum}_{k=0}^{K-1}
    \E\left[\Reg_\cA(\bx_k;\bg_{k,1},\dots,\bg_{k,M})
    \right]
  +\frac{rG}{\sqrt M}\,.
  \label{eq:generic-conversion}
\end{align}
\end{theorem}

The guarantee in \Cref{thm:generic-conversion} is stated in terms of the \emph{local} Frank-Wolfe gap over $C_\cA(\bx_I)\subseteq\cX\cap\Ball(\bx_I,r)$, which contains the output $\bx_I$. In \Cref{sec:projection,sec:lmo}, respectively, we relate this local gap to the Goldstein residual and the global Goldstein Frank-Wolfe gap.
We further remark that the additive $rG/\sqrt M$ term in \eqref{eq:generic-conversion} can be relaxed to depend on the variance of the stochastic gradients instead of the second moment. In particular, whenever $\sup_{\bx\in\cX}\E_\xi[\norm{\nabla F(\bx;\xi)-\nabla f(\bx)}^2]\le\sigma^2$,
it can be replaced by $r\sigma/\sqrt M$, which yields a slightly stronger result in low-noise regimes.

\section{Projection-based Algorithm}
\label{sec:projection}

\begin{algorithm2e}[t]
\caption{Online gradient descent with an anchor$(\bx_0,r,\alpha)$}
\label{alg:projected-online-learner}
\KwIn{anchor $\bx_0\in\cX$, radius $r>0$, step size $\alpha>0$.}

$\by_1\gets\bx_0$\;
\For{$m=1,\ldots,M$}{
  predict $\by_m$\;
  receive loss vector $\bell_m$\;
  $\bv_m\gets\proj_\cX(\by_m-\alpha\bell_m)$\;
  \eIf{$\norm{\bv_m-\bx_0}\le r$}{
    $\by_{m+1}\gets\bv_m$\;
  }{
    $\by_{m+1}\gets\bx_0+r\frac{\bv_m-\bx_0}{\norm{\bv_m-\bx_0}}$\;
  }
}
\Return{$\by_{M+1}$}\;
\end{algorithm2e}

In this section, we instantiate \Cref{alg:conversion} with a variant of online gradient descent
\citep{zinkevich2003online} as the online learner. Starting from
$\bx_0$, \Cref{alg:projected-online-learner} projects each gradient step onto $\cX$, and, if necessary, moves the
projected point toward $\bx_0$ so that it remains in $\Ball(\bx_0,r)$, which ensures the $r$-local OLO property.
To analyze \Cref{alg:projected-online-learner}, we build on the simple observation that OGD does not need projections onto $\cX\cap\Ball(\bx_0,r)$ but rather only that for any point
$\by\in\cX\cap\Ball(\bx_0,r)$, the update satisfies
$\norm{\by_{m+1}-\by}\le\norm{\by_m-\alpha\bell_m-\by}$. Hence the standard
regret bound holds over
$\cX\cap\Ball(\bx_0,r)$, as we show in
\Cref{lem:fejer-appendix,lem:projected-regret-appendix}.
Plugged into \Cref{thm:generic-conversion}, this yields a local Frank-Wolfe gap guarantee.

However, this alone is insufficient to control the Goldstein residual. Indeed,
\Cref{thm:generic-conversion} controls the local quantity
\(
\max_{\by\in\cX\cap\Ball(\bx,r)}
\inner{\bg}{\bx-\by}
\),
where $\bg$ is the average gradient returned by the conversion. This quantity is at least
\(
\min\{\eta\norm{\Gmap_{\eta,\cX}(\bx,\bg)}^2,\,
r\norm{\Gmap_{\eta,\cX}(\bx,\bg)}\}
\) (as we show in \Cref{lem:local-linearization-mapping}),
thus OGD can achieve $\norm{\Gmap_{\eta,\cX}(\bx,\bg)}\leq\eps$ 
only for
$\eta\gtrsim r/\eps$, and since the mapping norm is non-increasing in
$\eta$ by \Cref{lem:projection-residual}, this does not control the
smaller $\eta$ parameters needed for the Goldstein residual.

The key observation is that
\(
\norm{\bx-\proj_\cX(\bx-\eta\bg)}
=\eta\norm{\Gmap_{\eta,\cX}(\bx,\bg)}
\lesssim r
\),
thus taking $r\approx\delta$
ensures that
$\bg\in\partial_\delta f(\proj_\cX(\bx-\eta\bg))$,
rather than $\bg\in\partial_\delta f(\bx)$. This displacement introduces an additional error term which is controlled in expectation, so \Cref{alg:projection-expected} repeats this process $O(\log(1/\eps))$-many times to ensure the final residual bound, resulting in the following:

\begin{algorithm2e}[t]
\caption{Projection-based wrapper algorithm$(\bx_0,r,\eps,K,M,N,\alpha)$}
\label{alg:projection-expected}
\KwIn{initialization $\bx_0$, parameters $r,\eps,K,M,N,\alpha$.}

$n_\star\gets N$\;
$\cA:=\text{\Cref{alg:projected-online-learner}}(\cdot,r,\alpha)$\;
\For{$n=1,\ldots,N$}{
    $(\bx_n,(\bz_{n,m})_{m=1}^M,\widehat{\bg}_n)
    \gets
    \text{\Cref{alg:conversion}}(\bx_0,K,M,r,\cA)$\;
  $\bp_n\gets
  \proj_{\cX}\!\left(
    \bx_n-\frac{6r}{\eps}\widehat{\bg}_n
  \right)$\;
  \If{$\norm{\bx_n-\bp_n}\le4r$}{
    $n_\star\gets n$\;
    \textbf{break}\;
  }
}
\Return{$\bx^\mathrm{out}=\bp_{n_\star}$}\;
\end{algorithm2e}

\begin{theorem}
\label{thm:projection-expectation}
Let $\eps,\delta>0$. Assume $\bx_0\in\cX$ satisfies
$f(\bx_0)-\inf_{\bx\in\cX}f(\bx)\le\Delta<\infty$.
Run \Cref{alg:projection-expected} with
\(
  r=\delta/5,\,
  M=\left\lceil\left(\frac{96G}{\eps}\right)^2\right\rceil,\,
  \alpha=\frac{r}{G\sqrt M},\,
  K=\left\lceil\frac{48\Delta}{r\eps}\right\rceil,\,
  N=\left\lceil2\log\left(\frac{4G}{\eps}\right)\right\rceil
\).
Then, using
\(
  \wtO\!\left(\frac{\Delta G^2}{\delta\eps^3}\right)
\)
stochastic-gradient queries, it returns
$\bx^\mathrm{out}\in\cX$ satisfying
$\E[\Res^{f,\cX}_\delta(\bx^\mathrm{out})]\le\eps$.
\end{theorem}

The gradient oracle complexity of \Cref{thm:projection-expectation} is optimal (up to a logarithmic factor). Indeed, in the unconstrained setting as a particular case, the lower bounds of
\citet{arjevani2023lower,cutkosky2023} already imply an
$\Omega(\Delta G^2/(\delta\eps^3))$ rate for the Goldstein residual,
since for $\cX=\R^d$ the residual reduces to the standard Goldstein stationarity measure.
Complementarily, \Cref{cor:goldstein-res-lower} derives the same lower bound on a bounded constrained domain.
Moreover, as recently shown by \citet{kornowski2026complexity}, this rate cannot be improved even if the gradients are deterministic.
\section{Projection-Free Algorithm}
\label{sec:lmo}

In this section, we additionally assume that $\cX$ is compact with diameter
$D:=\max_{\bx,\by\in\cX}\norm{\bx-\by}<\infty$, and replace projections
by access to a linear minimization oracle
\[
  \LMO_{\cX}(\by)\in\argmin_{\bx\in\cX}\inner{\by}{\bx},
\]
which, in many applications, is substantially cheaper to evaluate
\citep{jaggi2010simple,lacoste2013block,jaggi2013,
combettes2021complexity}.
A natural projection-free choice for the OLO algorithm is Follow-the-Perturbed-Leader (FTPL)
\citep{kalai2005efficient,hazan2020faster}. However, in our conversion,
the linear loss depends on the learner's current prediction,
which invalidates the standard FTPL analysis.

Another natural candidate is Online Frank-Wolfe (OFW)
\citep{hazan2012projection}, which uses only one LMO call per round but
has regret much larger than the $O(\delta G\sqrt M)$ needed for our optimal rates. As observed by
\citet{kretzu2021revisiting}, OFW can be viewed as approximately
implementing Follow-the-Regularized-Leader (FTRL)
\citep{shalev2007primal} by taking a single Frank-Wolfe step toward the
minimizer of the current FTRL objective. We instead use $N$ iterations
of the Frank-Wolfe algorithm \citep{frank1956algorithm,jaggi2013} to approximate each quadratic
FTRL subproblem in \Cref{alg:lmo-ftrl}. Taking $N=4M$ preserves the
$O(\delta G\sqrt M)$ regret (\Cref{prop:lmo-online-main}) and hence the
optimal gradient oracle complexity of \Cref{thm:lmo-main}:

\begin{algorithm2e}[t]
\caption{Projection-free follow-the-regularized-leader$(\bx_0,\delta,\alpha,M,N)$}
\label{alg:lmo-ftrl}
\KwIn{Initialization $\bx_0$, radius $\delta$, step size $\alpha$,
budgets $M,N$.}

\For{$m=1,\ldots,M$}{
  $\bw_0\gets\bx_0$\;
  \For{$n=0,\ldots,N-1$}{
    $\bv_n
    \gets
    \LMO_{\cX}\!\left(
      \sum_{j=1}^{m-1}\bell_j
      +\alpha^{-1}(\bw_n-\bx_0)
    \right)$\;
    $\bw_{n+1}
    \gets
    \frac{n}{n+2}\bw_n
    +
    \frac{2}{n+2}\left[
      \left(1-\frac{\delta}{D}\right)\bx_0
      +\frac{\delta}{D}\bv_n
    \right]$\;
  }
  predict $\by_m\gets\bw_N$ and 
  obtain loss vector $\bell_m$\;
}
\end{algorithm2e}

\begin{theorem}
\label{thm:lmo-main}
Let $\eps>0$ and $0<\delta<D$. Assume $\bx_0\in\cX$ satisfies
$f(\bx_0)-\inf_{\bx\in\cX}f(\bx)\le\Delta<\infty$. Set
\(
  K=\left\lceil\frac{2\Delta D}{\delta\eps}\right\rceil,\,
  M=\left\lceil\left(\frac{6DG}{\eps}\right)^2\right\rceil,\,
  N=4M,\,
  \alpha=\frac{\delta}{G\sqrt M}
\).
Run \Cref{alg:conversion} with $r=\delta$ and
\Cref{alg:lmo-ftrl} as $\cA$.
Then, using
\(
  O\!\left(\frac{\Delta D^3G^2}{\delta\eps^3}\right)
\)
stochastic-gradient queries, it returns $\bx^\mathrm{out}\in\cX$ satisfying
$\E[\FW^{f,\cX}_\delta(\bx^\mathrm{out})]\le\eps$.
\end{theorem}

We later provide in
\Cref{thm:smooth-fw-lower} a lower bound for the Frank-Wolfe gap even for smooth objectives, while \Cref{cor:goldstein-fw-lower} yields the matching
$\Omega(\Delta D^3G^2/(\delta\eps^3))$ lower bound for the Goldstein Frank-Wolfe gap. Thus, the gradient oracle complexity of \Cref{thm:lmo-main} is optimal up to universal constants.
Notably, the analysis described above introduces an additional cost in terms of LMO calls. Faster Frank-Wolfe
variants can improve this complexity under stronger geometric
assumptions on $\cX$
\citep{garber2015faster,garber2016linearly}, while other projection-free
online methods use different oracle models
\citep{mhammedi2022efficient,pmlr-v291-mhammedi25a}. Extending our
approach in these directions is left for future work.

\section{Zeroth-Order Optimization}
\label{sec:zero-order}
In this section, instead of access to a stochastic-gradient
oracle, we assume access to stochastic component values
$F(\bx;\xi)$. We assume that each stochastic component
$F(\,\cdot\,;\xi):\R^d\to\R$ is $L(\xi)$-Lipschitz, with
\[
    \E_{\xi\sim\Xi}[L(\xi)^2]\le G^2.
\]

As is standard in this setting, we will make use of 
a two-point gradient estimator of the \emph{randomized smoothing}. To that end, the uniform smoothing over a ball of a function $f$ is defined as
\begin{align*}
f_\rho(\bx):=\E_{\bz\sim\Unif(\Ball(\mathbf{0},\rho))}\left[f(\bx+\bz)\right]\,.
\end{align*}
It is well known
(see \citealp{yousefian2012stochastic,duchi2012randomized}) that $f_\rho$ is
$G$-Lipschitz, is $O(\sqrt{d}G\rho^{-1})$-smooth, and satisfies
$|f(\bx)-f_\rho(\bx)|\le \rho G$ for every $\bx\in\R^d$.
To obtain zeroth-order results, we simply apply \Cref{alg:conversion} to $f_\rho$, by
replacing its stochastic-gradient query by a call to the gradient estimator of \Cref{alg:two-point}
with inputs $\bz_{k,m},\rho,\xi_{k,m}$.
The formal analysis, which is deferred to \Cref{app:zero-order-proofs}, yields the following results:

\begin{algorithm2e}[t]
\caption{Two-point gradient estimator$(\bz,\rho,\xi)$}
\label{alg:two-point}
\KwIn{query point $\bz\in\R^d$, smoothing radius $\rho>0$, stochastic sample $\xi$.}

sample $\bu\sim\Unif(\Sph^{d-1})$\;
query $F(\bz+\rho\bu;\xi)$ and $F(\bz-\rho\bu;\xi)$\;
\Return{
\(
\dfrac{d}{2\rho}
\bigl(F(\bz+\rho\bu;\xi)-F(\bz-\rho\bu;\xi)\bigr)\bu
\)
}\;
\end{algorithm2e}

\begin{theorem}
\label{thm:two-point-projection}
Let \(\eps,\delta>0\). Assume $\bx_0\in\cX$ satisfies
$f(\bx_0)-\inf_{\bx\in\cX}f(\bx)\le\Delta$, and set
$\rho:=\min\{\delta/2,\Delta/(2G)\}$.
Run \Cref{alg:projection-expected} on $f_\rho$ using
\Cref{alg:two-point}, with parameters chosen as in
\Cref{thm:projection-expectation}, with $(G,\delta,\Delta)$ there
replaced by $(4(2\pi)^{1/4}\sqrt dG,\delta/2,2\Delta)$.
Then, using
\(
  \wtO\!\left(
    \frac{\Delta G^2d}
         {\delta\eps^3}
  \right)
\)
function-value queries, it returns $\bx^\mathrm{out}\in\cX$ satisfying $  \E[\Res_\delta^{f,\cX}(\bx^\mathrm{out})]\le\eps$.
\end{theorem}

\begin{theorem}
\label{thm:two-point-lmo}
Let $\eps>0$. Suppose $\cX$ is compact with diameter $D$, and let
\(0<\delta<D\). Assume $\bx_0\in\cX$ satisfies
$f(\bx_0)-\inf_{\bx\in\cX}f(\bx)\le\Delta$, and set
$\rho:=\min\{\delta/2,\Delta/(2G)\}$.
Run \Cref{alg:conversion} on $f_\rho$ using
\Cref{alg:two-point} and \Cref{alg:lmo-ftrl} as $\cA$, with parameters
chosen as in \Cref{thm:lmo-main}, with $(G,\delta,\Delta)$ there
replaced by $(4(2\pi)^{1/4}\sqrt dG,\delta/2,2\Delta)$.
Then, using
\(
  O\!\left(
    \frac{\Delta D^3G^2d}{\delta\eps^3}
  \right)
\)
function-value queries, it returns $\bx^\mathrm{out}\in\cX$ satisfying $\E[\FW_\delta^{f,\cX}(\bx^\mathrm{out})]\le\eps$.
\end{theorem}

As previously mentioned, both of the results above improve over the query complexities of \citet{liu2024} by a factor of $\Omega(\sqrt{d})$.

\begin{remark}[Overall complexity and parallel computation] \label{rem: overall cmplx}
Although our zeroth-order algorithms attain the optimal dependence of the query complexity on the
dimension and stationarity parameters discussed above, it is not clear
whether their computational complexity is optimal in terms of the number of projections
or LMO calls. Note that as stated,
\Cref{thm:two-point-projection,thm:two-point-lmo} require
\(\wtO(\Delta G^2d/(\delta\eps^3))\) projections and
\(O(\Delta D^5G^4d^2/(\delta\eps^5))\) LMO calls, respectively.
However, as in \citet{kornowski2024}, these computational costs can be reduced by
averaging independent zeroth-order gradient estimators, a standard approach
for trading query complexity for parallel computation
\citep{duchi2018minimax,bubeck2019complexity}: At each gradient query, we
can average $b$ independent two-point estimators, thus reducing their
second moment from order $dG^2$ to $G^2(1+d/b)$, while using $b$ times
as many samples per round. Taking $b=d$ preserves the query complexity
up to constants, while making the projection/LMO complexity dimension-free, with
\(\wtO(\Delta G^2/(\delta\eps^3))\) projections or \(O(\Delta D^5G^4/(\delta\eps^5))\) LMO calls. In contrast, the guarantees of \citet{liu2024} retain an
explicit dependence on $d$ in their computational complexity, although with a better dependence on $\eps$. Determining the optimal tradeoff
between query and projection/LMO-oracle complexity remains open.

\end{remark}

\section{Lower Bounds}
\label{sec:lower-bounds}

We further show that the convergence rates obtained in this work are essentially optimal.
For simplicity, we prove lower bounds for (possibly-)randomized zero-respecting first-order algorithms, a class of algorithms which includes all algorithms of interest in this setting (including our first-order algorithms). Extending such a lower bound to any randomized algorithm is by now standard in the literature \citep{carmon2017lower,arjevani2023lower,cutkosky2023},  while still being rather technical. As this is not the main result in our paper but rather complementary to demonstrate the optimality of our algorithms, we chose to focus here on zero-respecting algorithms for the sake of conciseness.
We state the
lower bounds informally here and defer their formal statements and proofs to
\Cref{app:lower-bounds}, where we also define the class of zero-respecting algorithms.

We start by providing a lower bound in terms of the Frank-Wolfe gap.
Although the Frank-Wolfe gap is a standard stationarity measure for constrained
nonconvex optimization and is widely used to analyze projection-free methods
\citep{jaggi2013,lacoste2016convergence}, to our knowledge, no complexity lower bound for the Frank-Wolfe gap was previously known in the stochastic setting, even
for smooth objectives. Our first result provides such a lower bound.

\begin{theorem}[\Cref{thm:smooth-fw-lower}, informal]
Under mild parameter conditions, there is a dimension
$d=O(H\Delta D^2/\eps^2)$ such that for any zero-respecting first-order
algorithm there is an $H$-smooth stochastic
objective $f$ satisfying
$f(\mathbf 0)-\inf_{\bx\in\cX}f(\bx)\le\Delta$ for which achieving
expected Frank-Wolfe gap at most $\eps$
requires $\Omega\left(\frac{\Delta HG^2D^4}{\eps^4}\right)$
stochastic-gradient queries.
\end{theorem}

\Cref{thm:smooth-fw-lower} gives a tight lower bound, as the standard minibatched stochastic
Frank-Wolfe method \citep{reddi2016stochastic} achieves
the same stochastic-gradient
complexity under our stochastic first-order oracle.

As a consequence of the theorem above, we obtain lower bounds for the nonsmooth
Goldstein Frank-Wolfe gap and for the Goldstein residual which match our upper bounds.

\begin{corollary}[\Cref{cor:goldstein-fw-lower,cor:goldstein-res-lower}, informal]
Under the same setting as above, achieving expected $\delta$-Frank-Wolfe gap at most $\eps$ requires $\Omega\left(
        \frac{\Delta D^3G^2}{\delta\eps^3}
    \right)$
stochastic-gradient queries, and achieving expected $\delta$-Goldstein residual at most $\eps$ requires $\Omega\left(
        \frac{\Delta G^2}{\delta\eps^3}
    \right)$
stochastic-gradient queries.
\end{corollary}

For the Goldstein residual, the resulting rate recovers the known lower
bounds from unconstrained nonsmooth nonconvex optimization
\citep{arjevani2023lower,cutkosky2023}. Moreover,
\citet{kornowski2026complexity} recently showed that the same lower bound holds even
in the deterministic-gradient setting.

In the zeroth-order setting, the dependence on the dimension and on the
stationarity parameters $\delta,\eps$ is optimal separately
\citep{duchi2015optimal,cutkosky2023,kornowski2024}. Moreover,
\citet{zhang2026jointlowerboundszerothorder} recently established the
joint $\Omega(dG^2\Delta/(\delta\eps^3))$ lower bound for the Goldstein
residual, matching \Cref{thm:two-point-projection} up to logarithmic factors.
Their lower bound restricts queries to a bounded ball, a restriction that
appears removable by the standard soft-projection construction
\citep{carmon2017lower}. Extending their construction to the Goldstein
Frank-Wolfe gap appears possible and is left for future work.

\section{Analysis}

\subsection{Basic Geometric Properties}
\label{app:projection-geometry}

We collect a few simple facts about Euclidean projection and the
stationarity measures we use. We first record two elementary properties
of Euclidean projection.

\begin{fact}
\label{fact:projection}
Let $\cX\subseteq\R^d$ be nonempty, closed, and convex. For
$\bz\in\R^d$ and $\bp\in\cX$, we have
$\bp=\proj_{\cX}(\bz)$ if and only if
$\bz-\bp\in\mathcal N_{\cX}(\bp)$.
Moreover, $\proj_{\cX}$ is nonexpansive, namely
$\norm{\proj_{\cX}(\bz)-\proj_{\cX}(\bw)}
\le\norm{\bz-\bw}$ for every $\bz,\bw\in\R^d$.
\end{fact}

\begin{lemma}
\label{lem:projection-residual}
For every $\bx\in\cX$ and $\bg\in\R^d$, the map
$\eta\mapsto\norm{\Gmap_{\eta,\cX}(\bx,\bg)}$ is non-increasing, and
\begin{equation*}
  \sup_{\eta>0}\norm{\Gmap_{\eta,\cX}(\bx,\bg)}
  =
  \dist\bigl(-\bg,\mathcal{N}_\cX(\bx)\bigr).
\end{equation*}
\end{lemma}

\begin{proof}
The monotonicity follows from
\citet[Lemma~3]{gafni1982convergence}.
For any $\bn\in \mathcal{N}_\cX(\bx)$,
\Cref{fact:projection} gives
$\bx=\proj_{\cX}(\bx+\eta\bn)$. Hence, by nonexpansiveness,
\[
  \norm{\Gmap_{\eta,\cX}(\bx,\bg)}
  =
  \frac{1}{\eta}
  \norm{
    \proj_{\cX}(\bx-\eta\bg)
    -
    \proj_{\cX}(\bx+\eta\bn)
  }
  \le \norm{\bg+\bn}.
\]
Minimizing over $\bn\in \mathcal{N}_\cX(\bx)$ and then taking the supremum over
$\eta>0$ gives
\[
  \sup_{\eta>0}\norm{\Gmap_{\eta,\cX}(\bx,\bg)}
  \le
  \dist\bigl(-\bg,\mathcal{N}_\cX(\bx)\bigr).
\]
To prove the opposite inequality and thus equality,
it holds by monotonicity that
\[
  L:=
  \lim_{\eta\downarrow0}
  \norm{\Gmap_{\eta,\cX}(\bx,\bg)}
  =
  \sup_{\eta>0}\norm{\Gmap_{\eta,\cX}(\bx,\bg)}.
\]
The preceding bound shows that
$\Gmap_{\eta,\cX}(\bx,\bg)$ is bounded as $\eta\downarrow0$.
Hence, by Bolzano--Weierstrass, there exist a sequence
$\eta_j\downarrow0$ and $\bv\in\R^d$ such that
\(
  \Gmap_{\eta_j,\cX}(\bx,\bg)\to\bv.
\)
Denote
\(
  \bp_j:=\proj_{\cX}(\bx-\eta_j\bg).
\)
Then $\bp_j\to\bx$, and by \Cref{fact:projection},
\[
  \Gmap_{\eta_j,\cX}(\bx,\bg)-\bg
  =
  \frac{\bx-\eta_j\bg-\bp_j}{\eta_j}
  \in \mathcal{N}_\cX(\bp_j),
\]
since the normal cone is a cone. Thus, for every $\by\in\cX$, we have 
\(
  \inner{
    \Gmap_{\eta_j,\cX}(\bx,\bg)-\bg
  }{
    \by-\bp_j
  }
  \le 0.
\)
Taking $j\to\infty$ gives
\(
  \inner{\bv-\bg}{\by-\bx}\le 0
\) for every $\by\in\cX$,
and hence $\bv-\bg\in \mathcal{N}_\cX(\bx)$. Therefore
\[
  \dist\bigl(-\bg,\mathcal{N}_\cX(\bx)\bigr)
  \le \norm{\bv}
  =L.
\]
\end{proof}

\begin{lemma}
\label{lem:prelim-stability}
For every $\bx\in\cX$, $\bg,\bh\in\R^d$, and $\eta>0$,
\begin{equation}
  \norm{
    \Gmap_{\eta,\cX}(\bx,\bg)
    -
    \Gmap_{\eta,\cX}(\bx,\bh)
  }
  \le \norm{\bg-\bh}.
  \label{eq:mapping-stability}
\end{equation}
If $\cX$ is compact with diameter $D$, then
\begin{equation}
  \abs{
    \max_{\bu\in\cX}\inner{\bg}{\bx-\bu}
    -
    \max_{\bu\in\cX}\inner{\bh}{\bx-\bu}
  }
  \le D\norm{\bg-\bh},
  \qquad
  \max_{\bu\in\cX}\inner{\bg}{\bx-\bu}
  \le D\dist\bigl(-\bg,\mathcal{N}_\cX(\bx)\bigr).
  \label{eq:gap-stability}
\end{equation}
\end{lemma}

\begin{proof}
Lipschitzness of the mapping appears in
\citet[Lemma~4]{ghadimi2016accelerated}.
For the gap bounds,
\[
  \abs{\inner{\bg-\bh}{\bx-\bu}}
  \le D\norm{\bg-\bh},
\]
and, for every $\bn\in \mathcal{N}_\cX(\bx)$ and $\bu\in\cX$,
\[
  \inner{\bg}{\bx-\bu}
  \le
  \inner{\bg+\bn}{\bx-\bu}
  \le D\norm{\bg+\bn}.
\]
Taking the appropriate maximum and infimum proves
\eqref{eq:gap-stability}.
\end{proof}

\begin{lemma}
\label{lem:local-linearization-mapping}
For every $\bx\in\cX$, $\bg\in\R^d$, $r>0$, and $\eta>0$,
\[
\max_{\by\in\cX\cap\Ball(\bx,r)}
\inner{\bg}{\bx-\by}
\ge
\min\left\{
    \eta\norm{\Gmap_{\eta,\cX}(\bx,\bg)}^2,\,
    r\norm{\Gmap_{\eta,\cX}(\bx,\bg)}
\right\}.
\]
\end{lemma}

\begin{proof}
Let
\(
\bp:=\proj_{\cX}(\bx-\eta\bg).
\)
By \Cref{fact:projection},
\(
\bx-\eta\bg-\bp\in\mathcal N_{\cX}(\bp).
\)
Since $\bx\in\cX$, the definition of the normal cone gives
\[
\inner{\bx-\eta\bg-\bp}{\bx-\bp}\le0,
\]
and hence
\[
\inner{\bg}{\bx-\bp}
\ge
\frac{1}{\eta}\norm{\bx-\bp}^2
=
\eta\norm{\Gmap_{\eta,\cX}(\bx,\bg)}^2.
\]

If $\norm{\bx-\bp}\le r$, then
$\bp\in\cX\cap\Ball(\bx,r)$, so
\[
\max_{\by\in\cX\cap\Ball(\bx,r)}
\inner{\bg}{\bx-\by}
\ge
\eta\norm{\Gmap_{\eta,\cX}(\bx,\bg)}^2.
\]
Otherwise, by convexity,
\(
\bx+\frac{r}{\norm{\bp-\bx}}(\bp-\bx)
\in\cX\cap\Ball(\bx,r).
\)
Thus
\[
\max_{\by\in\cX\cap\Ball(\bx,r)}
\inner{\bg}{\bx-\by}
\ge
\frac{r}{\norm{\bp-\bx}}
\inner{\bg}{\bx-\bp} \ge
\frac{r}{\eta}\norm{\bp-\bx}
=
r\norm{\Gmap_{\eta,\cX}(\bx,\bg)}.
\]
The result follows.
\end{proof}
\subsection{Proof of \Cref{thm:generic-conversion}: Constrained OLO-to-Nonconvex Conversion}
\label{app:reduction-proofs}

Fix an epoch $k$ and denote
\(
  \bnabla_{k,m}:=\nabla f(\bz_{k,m})
\).
Since for every $m\in[M]$ we have $\bz_{k,m}\in \Ball(\bx_k,r)$, it follows that
\(
  \frac1M\sum_{m=1}^M\bnabla_{k,m}
  \in\partial_r f(\bx_k)\,,
\)
and thus the first inequality in \eqref{eq:generic-conversion} follows from the
definition of $\FW_r^{f,C_\cA(\bx_I)}(\bx_I)$.
Since
\(C_\cA(\bx_k)\subseteq\Ball(\bx_k,r)\), we have by Cauchy--Schwarz
\begin{align*}
&\max_{\by\in C_\cA(\bx_k)}
\sum_{m=1}^M
\inner{\bnabla_{k,m}}{\bx_k-\by}
\\
&\quad\le
\sum_{m=1}^M
\inner{\bg_{k,m}}{\bx_k-\by_{k,m}}
+
\max_{\by\in C_\cA(\bx_k)}
\sum_{m=1}^M
\inner{\bg_{k,m}}{\by_{k,m}-\by}
+r\left\|
  \sum_{m=1}^M
  \bigl(\bnabla_{k,m}-\bg_{k,m}\bigr)
\right\|.
\end{align*}

By unbiasedness and the second-moment bound,
\[
  \E\!\left[
    \norm{\bnabla_{k,m}-\bg_{k,m}}^2
    \,\middle|\,\bz_{k,m}
  \right]
  =
  \E\!\left[
    \norm{\bg_{k,m}}^2
    \,\middle|\,\bz_{k,m}
  \right]
  -\norm{\bnabla_{k,m}}^2
  \le G^2.
\]
Moreover, since each oracle call uses a fresh independent sample, the
errors $\bnabla_{k,m}-\bg_{k,m}$ have conditional mean zero, and hence
their cross terms vanish. Therefore
\(
  \E\left\|
    \sum_{m=1}^M
    \bigl(\bnabla_{k,m}-\bg_{k,m}\bigr)
  \right\|^2
  \le MG^2,
\)
and Jensen's inequality gives
\(
  \E\left\|
    \sum_{m=1}^M
    \bigl(\bnabla_{k,m}-\bg_{k,m}\bigr)
  \right\|
  \le G\sqrt M.
\)

By conditional unbiasedness of the oracle, the uniform choice of
$s_{k,m}$, and the line-integral identity,
\begin{align*}
\E\inner{\bg_{k,m}}{\bx_k-\by_{k,m}}
&=
\E\left[
  \int_0^1
  \inner{
    \nabla f\!\left(\bx_k+s(\by_{k,m}-\bx_k)\right)
  }{
    \bx_k-\by_{k,m}
  }\,ds
\right]
\\
&=
\E\bigl[f(\bx_k)-f(\by_{k,m})\bigr].
\end{align*}
Since $\bx_{k+1}=\by_{k,J_k}$ with $J_k\sim\Unif([M])$,
\[
  \E\left[
    \sum_{m=1}^M
    \inner{\bg_{k,m}}{\bx_k-\by_{k,m}}
  \right]
  =
  M\E\bigl[f(\bx_k)-f(\bx_{k+1})\bigr].
\]

Taking expectations in the first display and dividing by $M$ therefore
gives
\begin{align*}
&\E\left[
  \max_{\by\in C_\cA(\bx_k)}
  \inner{
    \frac1M\sum_{m=1}^M\bnabla_{k,m}
  }{
    \bx_k-\by
  }
\right]
\\
&\qquad\le
\E\bigl[f(\bx_k)-f(\bx_{k+1})\bigr]
+
\frac1M
\E\left[
  \Reg_\cA(\bx_k;\bg_{k,1},\dots,\bg_{k,M})
\right]
+
\frac{rG}{\sqrt M}.
\end{align*}
Finally, summing over $k=0,\ldots,K-1$ and dividing by $K$ gives
\begin{align*}
&\frac1K\sum_{k=0}^{K-1}
\E\left[
  \max_{\by\in C_\cA(\bx_k)}
  \inner{
    \frac1M\sum_{m=1}^M\bnabla_{k,m}
  }{
    \bx_k-\by
  }
\right]
\\
&\qquad\le
\frac{\E[f(\bx_0)-f(\bx_K)]}{K}
+
\frac{1}{KM}\sum_{k=0}^{K-1}
\E\left[
  \Reg_\cA(\bx_k;\bg_{k,1},\dots,\bg_{k,M})
\right]
+
\frac{rG}{\sqrt M}
\\
&\qquad\le
\frac{\Delta}{K}
+
\frac{1}{KM}\sum_{k=0}^{K-1}
\E\left[
  \Reg_\cA(\bx_k;\bg_{k,1},\dots,\bg_{k,M})
\right]
+
\frac{rG}{\sqrt M},
\end{align*}
where the second inequality uses
$f(\bx_K)\ge\inf_{\bx\in\cX}f(\bx)$. Since $I\sim\Unif(\{0,\ldots,K-1\})$, this proves the second inequality in
\eqref{eq:generic-conversion}.

\subsection{Projection-Based Application}
\label{app:projection-proofs}

\begin{lemma}
\label{lem:fejer-appendix}
For a point $\bx\in\cX$ and radius $r>0$, define the clipped projection $\Phi_{\bx,r}:\R^d\to\R^d$ as 
\begin{align*} 
    \Phi_{\bx,r}(\bv)=\bx+
    \min\set{1,\frac{r}{\norm{\proj_\cX(\bv)-\bx}}}
    (\proj_\cX(\bv)-\bx)\,.
\end{align*}
Then for any $\bv\in\R^d$ we have $\Phi_{\bx,r}(\bv)\in \cX\cap\Ball(\bx,r)$, and for any $\by\in \cX\cap\Ball(\bx,r)$,
\[
  \norm{\Phi_{\bx,r}(\bv)-\by}\le\norm{\bv-\by}.
\]
\end{lemma}
\begin{proof}
Let $\bv\in\R^d$ and $\bw:=\proj_\cX(\bv)\in\cX$. If \(\norm{\bw-\bx}\le r\), then \(\Phi_{\bx,r}(\bv)=\bw\in \cX\cap \Ball(\bx,r)\). Otherwise it holds that $\Phi_{\bx,r}(\bv)=(1-\frac{r}{\norm{\bw-\bx}})\bx
+\frac{r}{\norm{\bw-\bx}}\bw$,
and since \(\cX\) is convex and \(\bx,\bw\in\cX\), it holds that \(\Phi_{\bx,r}(\bv)\in \cX\). By construction \(\Phi_{\bx,r}(\bv)\in \Ball(\bx,r)\), so overall \(\Phi_{\bx,r}(\bv)\in \cX\cap\Ball(\bx,r)\), proving the first claim.

Furthermore, let $\by\in \cX\cap\Ball(\bx,r)$ and note that
\(
\Phi_{\bx,r}(\bv)
=
\proj_{\Ball(\bx,r)}\!\left(\proj_{\cX}(\bv)\right)\), 
so nonexpansiveness of the projections 
proves the second claim.
\end{proof}

\Cref{lem:projected-regret-appendix} is simply the regret guarantee of online gradient descent (see, e.g., \cite{hazan2016introduction}) instantiated over $\cX\cap\Ball(\bx,r)$, which we give here for completeness.

\begin{lemma}
\label{lem:projected-regret-appendix}
For any sequence of loss vectors $\bell_1,\ldots,\bell_M$, \Cref{alg:projected-online-learner} satisfies, for every $\by\in\cX\cap\Ball(\bx_0,r)$,
\[
  \sum_{m=1}^M\inner{\bell_m}{\by_m-\by}
  \le
  \frac{r^2}{2\alpha}
  +\frac{\alpha}{2}\sum_{m=1}^M\norm{\bell_m}^2.
\]
\end{lemma}
\begin{proof}
The update in \Cref{alg:projected-online-learner} is exactly $\Phi_{\bx_0,r}(\by_m-\alpha\bell_m)$.  Thus, by \Cref{lem:fejer-appendix},
\(
  \norm{\by_{m+1}-\by}^2
  \le
  \norm{\by_m-\alpha\bell_m-\by}^2.
\)
Expanding, rearranging, and summing over $m$ gives
\[
  \sum_{m=1}^M\inner{\bell_m}{\by_m-\by}
  \le
  \frac{\norm{\by_1-\by}^2}{2\alpha}
  +\frac{\alpha}{2}\sum_{m=1}^M\norm{\bell_m}^2.
\]
Since $\by_1=\bx_0$ and $\by\in \cX\cap\Ball(\bx_0,r)$,
$\norm{\by_1-\by}\le r$.
\end{proof}

\begin{proposition}
\label{prop:projection-base}
Let $r>0$ and $\eps_0>0$.
Run \Cref{alg:conversion} with $M\geq\left(\frac{16G}{\eps_0}\right)^2
\,,\,
  K\geq\frac{8\Delta}{r\eps_0}$
  and \(\cA\) given by
\Cref{alg:projected-online-learner} with step size
$\alpha=r/(G\sqrt M)$. Its outputs satisfy
\begin{equation}
  \E\left[\norm{\Gmap_{r/\epsilon_0,\cX}(\bx_I,\frac1M\sum_{m=1}^M\nabla f(\bz_{I,m}))}\right]\le\eps_0\,.
  \label{eq:projection-base-conclusion}
\end{equation}
\end{proposition}
\begin{proof}
For each epoch $k$, condition on the anchor $\bx_k$ and apply
\Cref{lem:projected-regret-appendix} with $\bx=\bx_k$ and
$\bell_m=\bg_{k,m}$. Taking expectations and using $\E\norm{\bg_{k,m}}^2\le G^2$ gives
\[
  \E\left[
    \max_{\by\in \cX\cap\Ball(\bx_k,r)}
    \sum_{m=1}^M
    \inner{\bg_{k,m}}{\by_{k,m}-\by}
  \right]
  \le
  \frac{r^2}{2\alpha}
  +\frac{\alpha MG^2}{2}.
\]
With $\alpha=r/(G\sqrt M)$, the RHS is $rG\sqrt M$.
Write \(\bnabla_I:=
  \frac1M\sum_{m=1}^M\nabla f(\bz_{I,m})\), then \Cref{thm:generic-conversion} gives
\[
  \E\left[
    \max_{\by\in \cX\cap\Ball(\bx_I,r)}
    \inner{\bnabla_I}{\bx_I-\by}
  \right]
    \le
  \frac{\Delta}{K}
  +\frac{2rG}{\sqrt M}
  \le
  \frac{r\eps_0}{4},
\]
where the last inequality uses the choices of $K$ and $M$. Let
\(
  \bp:=\proj_{\cX}(\bx_I-\frac{r}{\eps_0}\bnabla_I).
\)
By \Cref{fact:projection},
\[
  \inner{
    \bx_I-\frac{r}{\eps_0}\bnabla_I-\bp
  }{
    \bx_I-\bp
  }
  \le 0,
\]
and therefore
\[
  \inner{\bnabla_I}{\bx_I-\bp}
  \ge
  \frac{\eps_0}{r}\norm{\bx_I-\bp}^2
  =
  \frac{r}{\eps_0}
  \norm{\Gmap_{r/\eps_0,\cX}(\bx_I,\bnabla_I)}^2.
\]
If $\norm{\bx_I-\bp}\le r$, then
$\bp\in \cX\cap\Ball(\bx_I,r)$ and thus
\[
  \max_{\by\in \cX\cap\Ball(\bx_I,r)}
  \inner{\bnabla_I}{\bx_I-\by}
  \ge
  \frac{r}{\eps_0}
  \norm{\Gmap_{r/\eps_0,\cX}(\bx_I,\bnabla_I)}^2.
\]
Otherwise,
\[
  \by:=
  \bx_I+
  r\frac{\bp-\bx_I}{\norm{\bp-\bx_I}}
  \in \cX\cap\Ball(\bx_I,r),
\]
since $\bp,\bx_I\in\cX$ and
$r/\norm{\bp-\bx_I}\in(0,1)$. Therefore
\[
  \max_{\by\in \cX\cap\Ball(\bx_I,r)}
  \inner{\bnabla_I}{\bx_I-\by}
  \ge
  \frac{r}{\norm{\bp-\bx_I}}
  \inner{\bnabla_I}{\bx_I-\bp}
  \ge
  r\norm{\Gmap_{r/\eps_0,\cX}(\bx_I,\bnabla_I)}.
\]
So in either case,
\[
  \norm{\Gmap_{r/\eps_0,\cX}(\bx_I,\bnabla_I)}
  \le
  \sqrt{
    \frac{\eps_0}{r}
    \max_{\by\in \cX\cap\Ball(\bx_I,r)}
    \inner{\bnabla_I}{\bx_I-\by}
  }
  +
  \frac1r
  \max_{\by\in \cX\cap\Ball(\bx_I,r)}
  \inner{\bnabla_I}{\bx_I-\by}.
\]
Taking expectations and applying Jensen's inequality gives
\[
\begin{aligned}
  \E\norm{\Gmap_{r/\eps_0,\cX}(\bx_I,\bnabla_I)}
  \le
  \sqrt{\frac{\eps_0}{r}\frac{r\eps_0}{4}}
  +\frac{1}{r}\frac{r\eps_0}{4}
  =\frac34\eps_0
  \le\eps_0\,.
\end{aligned}
\]
\end{proof}

\begin{lemma}
\label{lem:projected-postprocessing}
Let $\cX$ be a nonempty closed convex set and let $\bx\in\cX$, $\bg\in\partial_r f(\bx)$, $\bv\in\R^d$. Denote
\(
  \bp:=\proj_{\cX}(\bx-\eta\bv).
\)
Then
\[
  \bg\in\partial_{r+\norm{\bp-\bx}}f(\bp),
  \qquad
  \dist\bigl(-\bg,\mathcal{N}_\cX(\bp)\bigr)
  \le
  \norm{\Gmap_{\eta,\cX}(\bx,\bv)}
  +\norm{\bg-\bv}.
\]
\end{lemma}

\begin{proof}
Since $\bg\in\partial_r f(\bx)$ and
\[
  \Ball(\bx,r)
  \subseteq
  \Ball\bigl(\bp,r+\norm{\bp-\bx}\bigr),
\]
the first claim follows. By \Cref{fact:projection},
\(
  \bx-\eta\bv-\bp\in \mathcal{N}_\cX(\bp).
\)
Since $\mathcal{N}_\cX(\bp)$ is a cone,
\[
  \Gmap_{\eta,\cX}(\bx,\bv)-\bv
  =
  \frac{\bx-\bp}{\eta}-\bv
  \in \mathcal{N}_\cX(\bp).
\]
Therefore,
\[
  \dist\bigl(-\bg,\mathcal{N}_\cX(\bp)\bigr)
  \le
  \norm{
    -\bg-\bigl(\Gmap_{\eta,\cX}(\bx,\bv)-\bv\bigr)
  }
  \le
  \norm{\Gmap_{\eta,\cX}(\bx,\bv)}
  +\norm{\bg-\bv}.
\]
\end{proof}

\begin{proof}[Proof of \Cref{thm:projection-expectation}]
For each trial $n$, denote
\(
  \bnabla_n:=\frac1M\sum_{m=1}^M\nabla f(\bz_{n,m})
\).
By unbiasedness, independence, and the second-moment bound, 
\(
  \E \norm{\bnabla_n-\widehat{\bg}_n}^2
  \le \frac{G^2}{M},
\)
so
\(
  \E \norm{\bnabla_n-\widehat{\bg}_n}
  \le\frac{G}{\sqrt M}\le\frac{\eps}{96}.
\)
By \Cref{prop:projection-base} with $\eps_0=\eps/6$ and
\Cref{lem:prelim-stability} we have
\[
  \E \norm{\Gmap_{6r/\eps,\cX}(\bx_n,\widehat{\bg}_n)}
  \le
  \E\norm{\Gmap_{6r/\eps,\cX}(\bx_n,\bnabla_n)}
  +\E \norm{\bnabla_n-\widehat{\bg}_n}
  \le \frac{\eps}{3},
\]
so Markov's inequality gives
\[
  \Prb\left(
    \norm{\Gmap_{6r/\eps,\cX}(\bx_n,\widehat{\bg}_n)}
    \le\frac{2\eps}{3}
  \right)
  \ge\frac12.
\]
By definition of $\Gmap_{6r/\eps,\cX}$,
\[
  \norm{\bp_n-\bx_n}
  =
  \frac{6r}{\eps}
  \norm{\Gmap_{6r/\eps,\cX}(\bx_n,\widehat{\bg}_n)},
\]
so each trial is accepted with probability at least $1/2$. Let $E$ denote the event that at least one trial is accepted. Since the trials are i.i.d.,
$$\Prb(E^c)
  \le 2^{-N}
  \le \frac{\eps}{4G}.
$$
On $E$,
$n_\star$ is the first accepted trial. As
$\bz_{n_\star,1},\dots,\bz_{n_\star,M}\in\Ball(\bx_{n_\star},r)$, we have
\(
  \bnabla_{n_\star}\in\partial_r f(\bx_{n_\star}).
\)
Since the trial is accepted,
$\norm{\bx^\mathrm{out}-\bx_{n_\star}}\le4r$, and hence
\(
  \bnabla_{n_\star}\in\partial_{5r}f(\bx^\mathrm{out})
  =\partial_\delta f(\bx^\mathrm{out})
\)
by the choice of $r$. Therefore, \Cref{lem:projected-postprocessing} gives
\[
  \Res_\delta^{f,\cX}(\bx^\mathrm{out})
  \leq\dist(-\bnabla_{n_\star},\mathcal{N}_\cX(\bx^\mathrm{out}))
  \le
  \norm{\Gmap_{6r/\eps,\cX}(\bx_{n_\star},\widehat{\bg}_{n_\star})}
  +\norm{\bnabla_{n_\star}-\widehat{\bg}_{n_\star}}
  \le
  \frac{2\eps}{3}+\norm{\bnabla_{n_\star}-\widehat{\bg}_{n_\star}}.
\]
Moreover,
\[
  \E\!\left[
    \norm{\bnabla_{n_\star}-\widehat{\bg}_{n_\star}}
    \,\middle|\,E
  \right]
  =
  \E\!\left[
    \norm{\bnabla_n-\widehat{\bg}_n}
    \,\middle|\,
    \norm{\bp_n-\bx_n}\le4r
  \right]
  \le
  \frac{\E\norm{\bnabla_n-\widehat{\bg}_n}}
       {\Prb(\norm{\bp_n-\bx_n}\le4r)}
  \le \frac{\eps}{48}.
\]
Thus
\[
  \E[\Res_\delta^{f,\cX}(\bx^\mathrm{out})\mid E]
  \le\frac{2\eps}{3}+\frac{\eps}{48}
  =\frac{11\eps}{16}.
\]
On $E^c$, $\bx^\mathrm{out}=\bp_N\in\cX$, and
\(
  \Res_\delta^{f,\cX}(\bx^\mathrm{out})\le G
\).
Therefore
\[
  \E\Res_\delta^{f,\cX}(\bx^\mathrm{out})
  \le
  \frac{11\eps}{16}
  +G\Prb(E^c)
  \le
  \frac{11\eps}{16}+\frac{\eps}{4}
  \le\eps.
\]
\end{proof}
\subsection{Projection-Free Application}
\label{app:lmo-proofs}

\begin{lemma}
\label{lem:homothetic-main}
For any $0<\delta\le D$ and $\bx\in\cX$,
\[
  \left(1-\frac{\delta}{D}\right)\bx+\frac{\delta}{D}\cX
  \subseteq
  \cX\cap\Ball(\bx,\delta).
\]
Moreover, for every $\bg\in\R^d$,
\[
  \max_{\by\in (1-\delta/D)\bx+(\delta/D)\cX}
  \inner{\bg}{\bx-\by}
  =
  \frac{\delta}{D}
  \max_{\bu\in\cX}\inner{\bg}{\bx-\bu}
\]
and thus
\[
  \FW_\delta^{f,(1-\delta/D)\bx+(\delta/D)\cX}(\bx)
  =
  \frac{\delta}{D}\FW_\delta^{f,\cX}(\bx).
\]
\end{lemma}
\begin{proof}
Write
\(
  \by
  =
  \left(1-\frac{\delta}{D}\right)\bx
  +\frac{\delta}{D}\bu
\)
for some $\bu\in\cX$. By convexity, $\by\in\cX$, and
\[
  \norm{\by-\bx}
  =
  \frac{\delta}{D}\norm{\bu-\bx}
  \le\delta.
\]
Also,
\(
  \bx-\by
  =
  \frac{\delta}{D}(\bx-\bu),
\)
which gives the second claim. Taking the minimum over $\bg\in\partial_\delta f(\bx)$ gives the last claim.
\end{proof}

\begin{proposition}
\label{prop:lmo-online-main}
Let $\cA_F$ be \Cref{alg:lmo-ftrl} with input $\bx_0$ and \(\alpha=\delta/(G\sqrt M),\,
  N=4M,\) and suppose 
\(
  \E\norm{\bell_m}^2\le G^2
\) for every $m\in[M]$. Then 
\[\E[\Reg_{\cA_F}(\bx_0;\bell_1,\dots,\bell_M)]
  \le 2\delta G\sqrt M.
\]
\end{proposition}

\begin{proof}
The proof combines the
Frank-Wolfe convergence rate \citep{jaggi2013} with the standard
quadratic FTRL regret bound \citep{orabona2019modern}, as we detail below.
Write
\[
  C:=\left(1-\frac{\delta}{D}\right)\bx_0+\frac{\delta}{D}\cX,
  \qquad
  Q_m(\by)
  :=
  \inner{\sum_{j=1}^{m-1}\bell_j}{\by}
  +\frac{1}{2\alpha}\norm{\by-\bx_0}^2,
\]
and let $\by_m^\star\in\argmin_{\by\in C}Q_m(\by)$.
By \Cref{lem:homothetic-main},
\[
  \left(1-\frac{\delta}{D}\right)\bx_0
  +\frac{\delta}{D}\bv_n
  \in
  \argmin_{\bs\in C}
  \inner{\nabla Q_m(\bw_n)}{\bs}\,,
\]
and therefore $\bw_n$ are the iterates of Frank-Wolfe (see, e.g., \citep[Algorithm~1]{jaggi2013}) on $Q_m$ over $C$. Since $Q_m$ is $\alpha^{-1}$-smooth and $C$ has diameter at most
$\delta$, the standard Frank-Wolfe guarantee
\citep[Theorem~1]{jaggi2013} gives
\[
  Q_m(\by_m)-Q_m(\by_m^\star)
  \le
  \frac{2\delta^2}{\alpha(N+2)}.
\]
Moreover, $Q_m$ is $\alpha^{-1}$-strongly convex so the LHS is lower bounded by 
\(
  \frac{1}{2\alpha}\norm{\by_m-\by_m^\star}^2\,,
\)
and thus
\[
  \norm{\by_m-\by_m^\star}
  \le \frac{2\delta}{\sqrt{N+2}}.
\]
The points $\by_m^\star$ are the FTRL iterates for the linear losses
$\by\mapsto\inner{\bell_m}{\by}$ with regularizer
$\psi(\by)=\frac12\norm{\by-\bx_0}^2$ and constant regularization coefficient
$\alpha$. Since $\psi$ is $1$-strongly convex and
$\min_{\by\in C}\psi(\by)=0$, the standard FTRL bound
\citep[Corollary~7.6]{orabona2019modern} gives, for every $\by\in C$,
\[
  \sum_{m=1}^M
  \inner{\bell_m}{\by_m^\star-\by}
  \le
  \frac{\norm{\by-\bx_0}^2}{2\alpha}
  +\frac{\alpha}{2}\sum_{m=1}^M\norm{\bell_m}^2.
\]
For every $\by\in C$, by using Cauchy--Schwarz, the fact that $\norm{\by-\bx_0}\le\delta$ and the guarantees above we arrive at
\[
\begin{aligned}
\sum_{m=1}^M\inner{\bell_m}{\by_m-\by}
&=
\sum_{m=1}^M\inner{\bell_m}{\by_m^\star-\by}
+\sum_{m=1}^M\inner{\bell_m}{\by_m-\by_m^\star}\\
&\le
\frac{\delta^2}{2\alpha}
+\frac{\alpha}{2}\sum_{m=1}^M\norm{\bell_m}^2
+\frac{2\delta}{\sqrt{N+2}}\sum_{m=1}^M\norm{\bell_m},
\end{aligned}
\]
Taking the maximum over
$\by\in C$ gives a bound on $\Reg_{\cA_F}(\bx_0;\bell_1,\dots,\bell_M)$, and then taking expectations and using
$\E\norm{\bell_m}\le(\E\norm{\bell_m}^2)^{1/2}\le G$,
\[
  \E[\Reg_{\cA_F}(\bx_0;\bell_1,\dots,\bell_M)]
  \le
  \frac{\delta^2}{2\alpha}
  +\frac{\alpha MG^2}{2}
  +\frac{2\delta MG}{\sqrt{N+2}}
  \le 2\delta G\sqrt M,
\]
where the last inequality uses
$\alpha=\delta/(G\sqrt M)$ and $N=4M$.
\end{proof}

\begin{proof}[Proof of \Cref{thm:lmo-main}]
Let $\bx^\mathrm{out}$ denote the first component returned by
\Cref{alg:conversion}. Then, by \Cref{thm:generic-conversion,prop:lmo-online-main},
\[
  \E\!\left[
    \FW_\delta^{f,(1-\delta/D)\bx^\mathrm{out}+(\delta/D)\cX}(\bx^\mathrm{out})
  \right]
  \le
  \frac{\Delta}{K}+\frac{3\delta G}{\sqrt M}.
\]
We use \Cref{lem:homothetic-main} to see that \Cref{alg:lmo-ftrl} is a $\delta$-local learner over $(1-\delta/D)\bx^\mathrm{out}+(\delta/D)\cX$, and use \Cref{lem:homothetic-main} again to write
\[
  \frac{\delta}{D}\E[\FW_\delta^{f,\cX}(\bx^\mathrm{out})]
  \le
  \frac{\Delta}{K}+\frac{3\delta G}{\sqrt M}
  \le \frac{\delta\eps}{D},
\]
which proves the claim.
\end{proof}

\subsection{Zeroth-Order Optimization}
\label{app:zero-order-proofs}

Our zeroth-order results rely on the following:

\begin{lemma}[\citealp{kornowski2024}]
\label{lem:smoothing-goldstein-inclusion}
For every \(\rho,\nu\ge 0\), we have
\(
    \partial_\nu f_\rho(\bx)
    \subseteq
    \partial_{\rho+\nu} f(\bx).
\)
Consequently, if \(\rho+\nu\le \delta\), then every \(\bg\in\partial_\nu f_\rho(\bx)\) is also in \(\partial_\delta f(\bx)\).
\end{lemma}

The following standard two-point estimator follows from
\citet{shamir2017optimal}, with the explicit constant in the
second-moment bound from \citet{lin2022gradient}.

\begin{lemma}
\label{lem:two-point-estimator}
For every $\rho>0$ and every possibly adaptive query point $\bz$,
if $\xi\sim\Xi$ is a fresh sample and $\widehat{\bg}$ is the output of
\Cref{alg:two-point}, then
\[
    \E[\widehat{\bg}\mid\bz]
    =
    \nabla f_\rho(\bz),
    \qquad
    \E[\|\widehat{\bg}\|^2\mid\bz]
    \le
    16\sqrt{2\pi}\,dG^2.
\]
\end{lemma}

In this section, when applying \Cref{alg:conversion} to $f_\rho$, we
replace its stochastic-gradient query by a call to \Cref{alg:two-point}
with inputs $\bz_{k,m},\rho,\xi_{k,m}$.

\begin{proof}[Proof of \Cref{thm:two-point-projection}]
Since $|f(\bx)-f_\rho(\bx)|\le \rho G$ for every $\bx$,
$f_\rho(\bx_0)-\inf_{\bx\in\cX}f_\rho(\bx)
\le \Delta+2G\rho\le 2\Delta$.
By \Cref{lem:two-point-estimator}, applying
\Cref{thm:projection-expectation} to $f_\rho$ with
$(G,\delta,\Delta)$ replaced by the parameters
$\bigl(4(2\pi)^{1/4}\sqrt d\,G,\delta/2,2\Delta\bigr)$ gives
\(
    \E\Res_{\delta/2}^{f_\rho,\cX}(\bx^\mathrm{out})
    \le\eps.
\)
By \Cref{lem:smoothing-goldstein-inclusion},
\(
    \partial_{\delta/2}f_\rho(\bx^\mathrm{out})
    \subseteq\partial_{\delta/2+\rho} f(\bx^\mathrm{out})
    \subseteq\partial_\delta f(\bx^\mathrm{out}),
\)
and therefore we have
\(
    \E\Res_\delta^{f,\cX}(\bx^\mathrm{out})\le\eps.
\)
The claimed complexity follows from
\Cref{thm:projection-expectation}.
\end{proof}

\begin{proof}[Proof of \Cref{thm:two-point-lmo}]
Since $|f(\bx)-f_\rho(\bx)|\le \rho G$ for every $\bx$,
$f_\rho(\bx_0)-\inf_{\bx\in\cX}f_\rho(\bx)
\le \Delta+2G\rho\le 2\Delta$.
By \Cref{lem:two-point-estimator}, applying
\Cref{thm:lmo-main} to $f_\rho$ with
$(G,\delta,\Delta)$ replaced by the parameters
$\bigl(4(2\pi)^{1/4}\sqrt d\,G,\delta/2,2\Delta\bigr)$ gives
\(
    \E\FW_{\delta/2}^{f_\rho,\cX}(\bx^\mathrm{out})\le\eps.
\)
By \Cref{lem:smoothing-goldstein-inclusion},
\(
    \partial_{\delta/2}f_\rho(\bx^\mathrm{out})
    \subseteq\partial_{\delta/2+\rho} f(\bx^\mathrm{out})
    \subseteq\partial_\delta f(\bx^\mathrm{out}),
\)
and therefore we have
\(
    \E\FW_\delta^{f,\cX}(\bx^\mathrm{out})\le\eps.
\)
The claimed complexity follows from \Cref{thm:lmo-main}.
\end{proof}
\subsection{Lower Bounds}
\label{app:lower-bounds}

As previously discussed, we formally prove our lower bounds for (possibly-)randomized zero-respecting first-order algorithms, which are defined as follows:

\begin{definition}[{\citealp[Definition~1]{arjevani2023lower}}]
A stochastic first-order algorithm is \emph{zero-respecting} if, for any
oracle and any realization of its randomness, its queries $\bx_t$ satisfy, for every
$t\ge1$,
\[
    \operatorname{supp}(\bx_t)
    \subseteq
    \bigcup_{i<t}\operatorname{supp}(\bg_i)\,,
\]
where $\bg_i$ is the stochastic gradient returned at $\bx_i$ and
$\operatorname{supp}(\bx):=\{j:x_j\ne0\}$.
\end{definition}

For $a\ge0$ and $\by\in\R^d$, we denote
\[
    \prog_a(\by)
    :=
    \max\bigl(\{i\in[d]:|y_i|>a\}\cup\{0\}\bigr)\,.
\]
Next, we define a \emph{probability-$p$ zero-chain}, and restate a few technical results from \citet{arjevani2023lower}.

\begin{definition}[{\citealp[Definition~2]{arjevani2023lower}}]
\label{def:probability-zero-chain}
A stochastic gradient function $g:\R^d\times\Xi\to\R^d$ is a
\emph{probability-$p$ zero-chain} if
\[
    \Pr_\xi\!\left(
        \exists\,\by\in\R^d:
        \prog_0(g(\by,\xi))
        =
        \prog_{1/4}(\by)+1
    \right)
    \le p
\]
and
\[
    \Pr_\xi\!\left(
        \exists\,\by\in\R^d:
        \prog_0(g(\by,\xi))
        >
        \prog_{1/4}(\by)+1
    \right)
    =0\,.
\]
\end{definition}

\begin{lemma}[{\citealp[Lemma~2]{arjevani2023lower}}]
\label{lem:hard-zero-chain}
There are universal constants $\Delta_0,\ell_1,\gamma_\infty>0$ such
that, for every $d\in\N$, there is a differentiable function
$\bar f_d:\R^d\to\R$ satisfying:
\begin{enumerate}
    \item $\bar f_d(\mathbf 0)-\inf_{\by}\bar f_d(\by)\le\Delta_0d$.
    \item The gradient of $\bar f_d$ is $\ell_1$-Lipschitz, i.e.,
    $\norm{\nabla\bar f_d(\by)-\nabla\bar f_d(\by')}
    \le\ell_1\norm{\by-\by'}$ for every $\by,\by'\in\R^d$.
    \item For every $\by\in\R^d$,
    $\norm{\nabla\bar f_d(\by)}_\infty\le\gamma_\infty$.
    \item For every $\by\in\R^d$ with $\prog_1(\by)<d$,
    $\norm{\nabla\bar f_d(\by)}
    \ge\abs{[\nabla\bar f_d(\by)]_{\prog_1(\by)+1}}>1$.
\end{enumerate}
\end{lemma}

\begin{lemma}[{\citealp[Lemma~8]{arjevani2023lower}}]
\label{lem:stochastic-zero-chain}
Let $d\in\N$, let $\bar f_d$ be as in \Cref{lem:hard-zero-chain}, and let
$p\in(0,1]$. There is a universal constant $\varsigma>0$ and a
differentiable stochastic function $\bar F_{d,p}(\cdot;\xi)$ such that:
\begin{enumerate}
    \item For every $\by\in\R^d$,
    $\E_\xi[\bar F_{d,p}(\by;\xi)]=\bar f_d(\by)$ and
    $\E_\xi[\nabla\bar F_{d,p}(\by;\xi)]=\nabla\bar f_d(\by)$.
    \item For every $\by\in\R^d$,
    $\E_\xi\norm{
        \nabla\bar F_{d,p}(\by;\xi)-\nabla\bar f_d(\by)
    }^2\le\varsigma^2/p$.
    \item The stochastic gradient $\nabla\bar F_{d,p}$ is a
    probability-$p$ zero-chain.
\end{enumerate}
\end{lemma}

\begin{lemma}[{\citealp[Lemma~1]{arjevani2023lower}}]
\label{lem:zero-chain-progress}
Let $g(\by,\xi)$ be a probability-$p$ zero-chain gradient estimator for
$\bar f:\R^d\to\R$, and let the oracle return
$(\bar f(\by),g(\by,\xi))$. If
$\by_1,\by_2,\ldots\in\R^d$ are the queries of a zero-respecting algorithm
applied to this oracle, then, with probability at least $1-\delta$,
\[
    \prog_0(\by_t)<d
    \qquad
    \text{for every }
    t\le\frac{d-\log(1/\delta)}{2p}\,.
\]
\end{lemma}

We first rescale these results to the parameters used in our lower bound.

\begin{lemma}
\label{lem:scaled-zero-chain}
There are universal constants $c,C>0$ such that the following holds.
Let $L,\Delta,\sigma,\eta>0$ satisfy
$\eta\le c\min\{\sigma,\sqrt{L\Delta}\}$.
Then there exists a dimension $d\le C L\Delta/\eta^2$ for which there is a
stochastic function $F_0:\R^d\times\Xi\to\R$, differentiable in its first argument, with mean
$f_0(\bx):=\E_\xi[F_0(\bx;\xi)]$, such that $f_0$ is $L$-smooth,
$f_0(\mathbf 0)-\inf_{\bx\in\R^d}f_0(\bx)\le\Delta$, and, for every
$\bx\in\R^d$,
$\norm{\nabla f_0(\bx)}\le C\sqrt{L\Delta}$,
$\E_\xi[\nabla F_0(\bx;\xi)]=\nabla f_0(\bx)$, and
$\E_\xi\norm{\nabla F_0(\bx;\xi)-\nabla f_0(\bx)}^2\le\sigma^2$.
Moreover, any randomized zero-respecting algorithm making
$T\le cL\Delta\sigma^2/\eta^4$ queries and returning $\bx^{\mathrm{out}}$
with
$\operatorname{supp}(\bx^{\mathrm{out}})
\subseteq\bigcup_{t\le T}\operatorname{supp}
(\nabla F_0(\bx_t;\xi_t))$
satisfies
\[
    \Pr\!\left(
        \exists j\in[d]:
        |x_j^{\mathrm{out}}|\le C\frac{\eta}{L},
        \quad
        |[\nabla f_0(\bx^{\mathrm{out}})]_j|\ge\eta
    \right)
    \ge\frac12\,.
\]
\end{lemma}

\begin{proof}
Let $\Delta_0,\ell_1,\gamma_\infty,\varsigma$ be as in
\Cref{lem:hard-zero-chain,lem:stochastic-zero-chain}, and set
\[
    \lambda:=\frac{\ell_1\eta}{L},
    \qquad
    d:=\left\lfloor
        \frac{L\Delta}{2\Delta_0\ell_1\eta^2}
    \right\rfloor,
    \qquad
    p:=\frac{\varsigma^2\eta^2}{\sigma^2}\,.
\]
For sufficiently small $c$ and large $C$,
$2\le d\le C L\Delta/\eta^2$ and $p\le1$. Define, using the
functions from \Cref{lem:hard-zero-chain,lem:stochastic-zero-chain},
\[
    f_0(\bx)
    :=
    \frac{L\lambda^2}{\ell_1}
    \bar f_d(\bx/\lambda),
    \qquad
    F_0(\bx;\xi)
    :=
    \frac{L\lambda^2}{\ell_1}
    \bar F_{d,p}(\bx/\lambda;\xi)\,.
\]
By \Cref{lem:hard-zero-chain} and $L\lambda/\ell_1=\eta$, $f_0$ is
$L$-smooth,
$f_0(\mathbf 0)-\inf_{\bx\in\R^d} f_0(\bx)
\le (L\lambda^2/\ell_1)\Delta_0d\le\Delta$, and
$\norm{\nabla f_0(\bx)}
\le\eta\gamma_\infty\sqrt d\le C\sqrt{L\Delta}$.
By \Cref{lem:stochastic-zero-chain}, $F_0$ is unbiased and
$\E_\xi\norm{\nabla F_0(\bx;\xi)-\nabla f_0(\bx)}^2
\le\eta^2\varsigma^2/p=\sigma^2$.
Appending $\bx^{\mathrm{out}}/\lambda$ as query $T+1$, and decreasing
$c$ if necessary,
\[
    T+1
    \le
    c\frac{L\Delta\sigma^2}{\eta^4}+1
    \le
    \frac{d-\log 2}{2p}\,.
\]
Hence \Cref{lem:zero-chain-progress} gives
$\prog_0(\bx^{\mathrm{out}}/\lambda)<d$ with probability at least $1/2$.
Now for this event, define 
$j:=\prog_1(\bx^{\mathrm{out}}/\lambda)+1$. Since
$\prog_1(\bx^{\mathrm{out}}/\lambda)<d$, we have $j\in[d]$, and by the
definition of $\prog_1$,
\[
    |x_j^{\mathrm{out}}|
    \le
    \lambda
    =
    \frac{\ell_1\eta}{L}
    \le
    C\frac{\eta}{L}\,.
\]
Moreover, by \Cref{lem:hard-zero-chain} and
$L\lambda/\ell_1=\eta$,
\[
    |[\nabla f_0(\bx^{\mathrm{out}})]_j|
    =
    \eta
    |[\nabla\bar f_d(\bx^{\mathrm{out}}/\lambda)]_j|
    >
    \eta\,.
\]
\end{proof}

To convert the hard coordinate in \Cref{lem:scaled-zero-chain} into
a Frank-Wolfe lower bound, we add a quadratic radial term to the hard
function. \Cref{lem:scaled-zero-chain} yields the desired gap near the origin,
while the radial term ensures a large gap away from the origin.

\begin{lemma}
\label{lem:hard-fw-gap}
Let $D>0$, let $f_0:\R^d\to\R$ be differentiable with
$\norm{\nabla f_0(\bx)}\le M$ for every $\bx\in\R^d$, and let
$0<\eta\le M$. Define
$f(\bx):=f_0(\bx)+(4M/D)\norm{\bx}^2$ and
$\cX:=\Ball(\mathbf 0,D/2)$. If $\bx\in\cX$ has a coordinate $j$ such that
$|x_j|\le D\eta/(16M)$ and
$|[\nabla f_0(\bx)]_j|\ge\eta$, then
$\FW^{f,\cX}(\bx)\ge D\eta/8$.
\end{lemma}

\begin{proof}
If $\norm{\bx}\le D/4$, then by choosing $-\frac D2\frac{\nabla f(\bx)}{\norm{\nabla f(\bx)}}\in\cX$ we have
\begin{align*}
\FW^{f,\cX}(\bx)
&\ge
\left\langle
\nabla f(\bx),
\bx+\frac D2\frac{\nabla f(\bx)}{\norm{\nabla f(\bx)}}
\right\rangle\\
&=
\left\langle\nabla f(\bx),\bx\right\rangle
+
\frac D2\norm{\nabla f(\bx)}\\
&\ge
-\norm{\nabla f(\bx)}\norm{\bx}
+
\frac D2\norm{\nabla f(\bx)}
\\&\ge
\frac D4\norm{\nabla f(\bx)}\\
&\ge
\frac D4\abs{[\nabla f(\bx)]_j}\\
&\ge
\frac D4
\left(
\abs{[\nabla f_0(\bx)]_j}
-
\frac{8M}{D}|x_j|
\right)\\
&\ge
\frac D4
\left(
\eta-\frac{8M}{D}\frac{D\eta}{16M}
\right)
=
\frac{D\eta}{8}\,.
\end{align*}
Otherwise, choose $-D\bx/(2\norm{\bx})\in\cX$ to get
\begin{align*}
\FW^{f,\cX}(\bx)
&\ge
\left\langle\nabla f(\bx),\bx+D\frac{\bx}{2\norm{\bx}}\right\rangle
\\
&\ge
\left(\frac D2+\norm{\bx}\right)
\left\langle
\nabla f(\bx),\frac{\bx}{\norm{\bx}}
\right\rangle\\
&=
\left(\frac D2+\norm{\bx}\right)
\left(
\left\langle
\nabla f_0(\bx),\frac{\bx}{\norm{\bx}}
\right\rangle
+\frac{8M}{D}\norm{\bx}
\right)\\
&\ge
\left(\frac D2+\norm{\bx}\right)
\left(\frac{8M}{D}\norm{\bx}-M\right)\\
&\ge
\frac{DM}{2}\\
&\ge
\frac{D\eta}{8}\,.
\end{align*}
\end{proof}

\begin{theorem}
\label{thm:smooth-fw-lower}
There are universal constants $c,C,C_d>0$ such that the following holds.
Let $H,\Delta,G,D,\eps>0$ satisfy
\[
    \Delta\le cHD^2,
    \qquad
    \eps\le cD\sqrt{H\Delta},
    \qquad
    G\ge C\sqrt{H\Delta}\,.
\]
Then there exists a dimension $d\le C_d H\Delta D^2/\eps^2$ for which
there are differentiable stochastic components $F(\cdot;\xi)$ whose mean $f:\R^d\to\R$ is $H$-smooth and, for $\cX:=\Ball(\mathbf 0,D/2)$, we have, for every $\bx\in\cX$,
\[
    f(\mathbf 0)-\inf_{\bx\in\cX}f(\bx)\le\Delta,
    \qquad
    \E_\xi[\nabla F(\bx;\xi)]=\nabla f(\bx),
    \qquad
    \E_\xi\norm{\nabla F(\bx;\xi)}^2\le G^2
    \,.
\]
Moreover, any randomized zero-respecting first-order algorithm
making queries $\bx_1,\ldots,\bx_T\in\cX$ and returning
$\bx^{\mathrm{out}}\in\cX$ satisfying
\[
    \operatorname{supp}(\bx^{\mathrm{out}})
    \subseteq
    \bigcup_{t\le T}\operatorname{supp}
    \bigl(\nabla F(\bx_t;\xi_t)\bigr),
    \qquad
    \E[\FW^{f,\cX}(\bx^{\mathrm{out}})]\le\eps,
\]
must have
\[
    T
    =
    \Omega\!\left(
        \frac{\Delta HG^2D^4}{\eps^4}
    \right)\,.
\]
\end{theorem}

\begin{proof}
Set $\eta:=32\eps/D$ and apply \Cref{lem:scaled-zero-chain} with
$L=H/2$ and $\sigma=G/2$. Let $C_{\rm zc}$ denote its universal
constant $C$, set $M:=C_{\rm zc}\sqrt{H\Delta}$, and define
\[
    F(\bx;\xi)
    :=
    F_0(\bx;\xi)+\frac{4M}{D}\norm{\bx}^2\,.
\]
For sufficiently small $c$ and large $C,C_d$, the assumptions of
\Cref{lem:scaled-zero-chain} hold, and its dimension satisfies the claimed bound.
Thus $f$ is $H$-smooth and we have that
\begin{align*}
    f(\mathbf 0)-\inf_{\bx\in\cX}f(\bx)
    &\le f_0(\mathbf 0)-\inf f_0
    \le\Delta
\end{align*}
and \begin{align*}
    \E_\xi\norm{\nabla F(\bx;\xi)}^2
    &=
    \norm{\nabla f(\bx)}^2
    +
    \E_\xi\norm{
        \nabla F_0(\bx;\xi)-\nabla f_0(\bx)
    }^2
    \le G^2\,.
\end{align*}
Given a zero-respecting algorithm for $F$, subtracting
$8M\bx/D$ from each oracle response simulates the $F_0$ oracle. Inductively,
for every $t$,
\[
    \operatorname{supp}(\bx_t)
    \subseteq
    \bigcup_{s<t}\operatorname{supp}
    \bigl(\nabla F_0(\bx_s;\xi_s)\bigr),
    \qquad
    \operatorname{supp}\bigl(\nabla F(\bx_t;\xi_t)\bigr)
    \subseteq
    \bigcup_{s\le t}\operatorname{supp}
    \bigl(\nabla F_0(\bx_s;\xi_s)\bigr)\,.
\]
Thus the simulated algorithm is zero-respecting with $\operatorname{supp}(\bx^{\mathrm{out}})
\subseteq\bigcup_{t\le T}\operatorname{supp}
(\nabla F_0(\bx_t;\xi_t))$. Hence, for a sufficiently small universal constant $c_0>0$, if
\[
    T
    \le
    c_0\frac{\Delta HG^2D^4}{\eps^4},
\]
then \Cref{lem:scaled-zero-chain,lem:hard-fw-gap} give
\[
    \E[\FW^{f,\cX}(\bx^{\mathrm{out}})]
    \ge
    \frac12\cdot\frac{D\eta}{8}
    =2\eps\,.
\]
\end{proof}

\begin{lemma}
\label{lem:smooth-goldstein-comparison}
Let $f$ be $H$-smooth and let $\cX$ be compact with diameter $D$. Then,
for every $\bx\in\cX$ and $\delta\ge0$,
\[
    \FW^{f,\cX}(\bx)-HD\delta
    \le
    \FW^{f,\cX}_\delta(\bx)
    \le
    D\Res^{f,\cX}_\delta(\bx)\,.
\]
\end{lemma}

\begin{proof}
For every $\bg\in\partial_\delta f(\bx)$, writing
$\bg=\sum_i\alpha_i\nabla f(\by_i)$ with
$\by_i\in\Ball(\bx,\delta)$ gives
\[
    \norm{\bg-\nabla f(\bx)}
    \le
    \sum_i\alpha_i
    \norm{\nabla f(\by_i)-\nabla f(\bx)}
    \le H\delta\,.
\]
Then by \Cref{lem:prelim-stability},
\begin{align*}
\FW^{f,\cX}(\bx)-HD\delta
&\le
\min_{\bg\in\partial_\delta f(\bx)}
\left\{
\max_{\bu\in\cX}\inner{\bg}{\bx-\bu}
+
D\norm{\bg-\nabla f(\bx)}
-
HD\delta
\right\}\\
&\le
\min_{\bg\in\partial_\delta f(\bx)}
\max_{\bu\in\cX}\inner{\bg}{\bx-\bu}\\
&=
\FW^{f,\cX}_\delta(\bx)\\
&\le
D\min_{\bg\in\partial_\delta f(\bx)}
\dist(-\bg,N_{\cX}(\bx))\\
&=
D\Res^{f,\cX}_\delta(\bx)\,.
\end{align*}
\end{proof}

The smooth Frank-Wolfe lower bound immediately yields the corresponding
lower bound for the Goldstein Frank-Wolfe gap, as we formalize in the corollary below.

\begin{corollary}
\label{cor:goldstein-fw-lower}
There are universal constants $c,C,C_d>0$ such that the following holds.
Let $\Delta,G,D,\eps>0$ and $0<\delta<D$ satisfy
\[
    C\frac{\eps\delta}{D}\le\Delta\le c\frac{\eps D}{\delta},
    \qquad
    G^2\ge C\frac{\eps\Delta}{D\delta}\,.
\]
Then there exists a dimension $d\le C_d\Delta D/(\delta\eps)$ for which
there are differentiable stochastic components $F(\cdot;\xi)$
with mean $f$ such that, for $\cX:=\Ball(\mathbf 0,D/2)$,
$f$ is $\eps/(D\delta)$-smooth and the initial-suboptimality and
stochastic-gradient bounds in \Cref{thm:smooth-fw-lower} hold.
Every randomized zero-respecting first-order algorithm satisfying the
output-support condition of \Cref{thm:smooth-fw-lower} requires
\[
    \Omega\!\left(
        \frac{\Delta D^3G^2}{\delta\eps^3}
    \right)
\]
stochastic-gradient queries to return $\bx^{\mathrm{out}}\in\cX$ with
$\E[\FW^{f,\cX}_\delta(\bx^{\mathrm{out}})]\le\eps$.
\end{corollary}

\begin{proof}
Apply \Cref{thm:smooth-fw-lower} with
$H=\eps/(D\delta)$ and ordinary Frank-Wolfe target $2\eps$. By
\Cref{lem:smooth-goldstein-comparison},
\[
    \FW^{f,\cX}_\delta(\bx)
    \ge
    \FW^{f,\cX}(\bx)-\eps\,.
\]
\end{proof}

Finally, the Goldstein Frank-Wolfe lower bound in turn yields the corresponding
lower bound for the Goldstein residual.

\begin{corollary}
\label{cor:goldstein-res-lower}
There are universal constants $c,C,C_d>0$ such that the following holds.
Let $\Delta,G,D,\eps>0$ and $0<\delta<D$ satisfy
\[
    C\eps\delta\le\Delta\le c\frac{\eps D^2}{\delta},
    \qquad
    G^2\ge C\frac{\eps\Delta}{\delta}\,.
\]
Then there exists a dimension $d\le C_d\Delta/(\delta\eps)$ for which
there are differentiable stochastic components $F(\cdot;\xi)$
with mean $f$ such that, for $\cX:=\Ball(\mathbf 0,D/2)$,
$f$ is $\eps/\delta$-smooth and the initial-suboptimality and
stochastic-gradient bounds in \Cref{thm:smooth-fw-lower} hold.
Every randomized zero-respecting first-order algorithm satisfying the
output-support condition of \Cref{thm:smooth-fw-lower} requires
\[
    \Omega\!\left(
        \frac{\Delta G^2}{\delta\eps^3}
    \right)
\]
stochastic-gradient queries to return $\bx^{\mathrm{out}}\in\cX$ with
$\E[\Res^{f,\cX}_\delta(\bx^{\mathrm{out}})]\le\eps$.
\end{corollary}

\begin{proof}
Apply \Cref{cor:goldstein-fw-lower} with target $D\eps$. By
\Cref{lem:smooth-goldstein-comparison},
$\FW^{f,\cX}_\delta(\bx)\le D\Res^{f,\cX}_\delta(\bx)$.
\end{proof}
\section{Discussion}
In this paper, we provided the first dimension-free algorithms for first-order constrained optimization of
arbitrary stochastic Lipschitz functions, without smoothness or convexity
assumptions, and showed that their rates are optimal for (possibly-)randomized zero-respecting algorithms. Our approach also extends to
the zeroth-order setting in \Cref{sec:zero-order} and closes the upper-bound gap
left open by \citet{liu2024}. As discussed in \Cref{sec:lower-bounds}, the recent joint lower bound of
\citet{zhang2026jointlowerboundszerothorder} shows the optimality of our
zeroth-order rate for the Goldstein residual already in the unconstrained case.

We leave several directions open for future work. 
A natural
direction is to further assess the overall computational complexity of this setting (see \Cref{rem: overall cmplx}), and to
extend the framework in this work more generally to composite optimization problems as in \citet{chen2025composite}. Finally,
similarly to \citet{liu2024}, our guarantees control the expectation of a stationarity
measure (in our case, a stronger one, as discussed in Remark~\ref{rem: stronger stat}), whereas obtaining analogous high-probability
guarantees should be possible using standard amplification techniques, which we do not pursue here. We also leave open formally extending the lower
bounds in \Cref{app:lower-bounds} from randomized zero-respecting algorithms to arbitrary randomized algorithms, which seems possible using known reduction techniques \citep{carmon2017lower}.

Lastly, we note that there exist practical drawbacks to our algorithms. Most notably, as is
common in nonsmooth nonconvex optimization, our guarantees require several
parameters to be tuned at different scales, making the methods less attractive from a practical standpoint. It would therefore be interesting to find simpler
algorithms that retain the same guarantees. One possible route is through a
reduction of nonsmooth nonconvex optimization to static regret of the type
proposed by \citet{chen2024open}, which could potentially lead to a substantially
simpler algorithmic framework.

\subsection*{Statement on use of AI}
The results in this paper were developed through a highly interactive process
involving AI models, specifically, GPT Sol 5.6 and GPT Astra. The authors began with a roadmap of the results, and formalized and generalized them with the help of AI.

Every mathematical argument appearing in the paper was independently re-derived, rewritten and simplified by the authors. The authors
take full responsibility for the correctness and presentation of all material
in this manuscript.

\newpage
\bibliographystyle{plainnat}
\bibliography{references}

\end{document}